\documentclass[18pt]{siamart220329}

\usepackage[english]{babel}
\usepackage{color}
\usepackage[letterpaper,top=2cm,bottom=2cm,left=3cm,right=3cm,marginparwidth=1.75cm]{geometry}
\usepackage{algorithm}
\usepackage{algpseudocode}
\usepackage{comment}
\usepackage{indentfirst}
\usepackage{color}
\usepackage{graphicx}
\usepackage{subcaption}
\usepackage{zymacros}
\usepackage[normalem]{ulem}
\usepackage{hyperref}
\usepackage{mathtools}

\newcommand{\yx}[1]{\textcolor{red}{YX: #1}}
\newcommand{\RL}[1]{\textcolor{cyan}{RL: #1}}

\usepackage{amsfonts}
\usepackage{enumerate}
\newcommand\expt{\ensuremath{\mathbb{E}}}
\providecommand{\norm}[1]{\lVert#1\rVert}
\usepackage[colorinlistoftodos]{todonotes}

\title{Multiscale Neural Networks for Approximating Green's Functions}

\author{Wenrui Hao\thanks{Department of Mathematics, The Pennsylvania State University, University Park, State College, PA, USA,
  (\texttt{\email{wxh64@psu.edu},\email{yxy5498@psu.edu}}). Research supported by National Institute of General Medical Sciences through grant
1R35GM146894.} \and Rui Peng Li\thanks{Center  for Applied  Scientific Computing,
    Lawrence  Livermore National  Laboratory,  P. O.  Box 808,  L-561,
    Livermore,   CA  94551   {(\texttt{\email{li50@llnl.gov}})}.  This   work  was
    performed under the  auspices of the U.S. Department  of Energy by
    Lawrence    Livermore   National    Laboratory   under    Contract
    DE-AC52-07NA27344 (LLNL-JRNL-870805) and was supported by the LLNL-LDRD program under 
    Project No. 24-ERD-033.}\and Yuanzhe Xi\thanks{Department of Mathematics, Emory University, Atlanta, GA, USA, (\texttt{\email{yxi26@emory.edu},\email{tianshi.xu@emory.edu}}). Research of Y. Xi is supported by NSF DMS-2338904. Research of T. Xu is supported by LLNL-LDRD program under 
    Project No. 24-ERD-033.} \and Tianshi Xu\footnotemark[3] 
    \and Yahong Yang\footnotemark[1]}

\headers{MSNN for Approximating Green's Functions}{W. Hao, R. Li, Y. Xi, T. Xu and Y. Yang} 

\begin{document}
\maketitle

\begin{abstract}
Neural networks (NNs) have been widely used to solve partial differential equations (PDEs) in 
the applications of
physics, biology, and engineering. One effective approach for solving PDEs with a fixed differential operator is learning Green's functions. However, Green's functions are notoriously difficult to learn due to their poor regularity, which typically requires larger NNs and longer training times. In this paper, we address these challenges by leveraging multiscale NNs to learn Green's functions. Through theoretical analysis using multiscale Barron space methods and experimental validation, we show that the multiscale approach significantly reduces the necessary NN size and accelerates training.
\end{abstract}


\begin{keywords} 
Green's function, Barron space, multiscale neural networks, PDEs, domain decomposition, fast solver
\end{keywords}

\begin{AMS}
65N55, 65N80, 68T07
\end{AMS}

\section{Introduction}
In recent years, the application of neural networks (NNs) to solving Partial Differential Equations (PDEs) has spurred the development of a range of innovative methodologies. These methods can generally be classified into two main categories: function learning and operator learning. The function learning methods, such as Physics-Informed NNs (PINNs) \cite{raissi2019physics}, the Deep Ritz Method \cite{weinan2018deep}, the Deep Galerkin Method \cite{sirignano2018dgm}, and approaches based on random features \cite{chen2022bridging,dong2023method,sun2024local}, utilize NNs to directly approximate the solutions using specifically designed loss functions.

A significant limitation of the function learning approaches is the need to train a separate NN for each PDE if it differs from the one the network was originally trained on.
On the other hand, the operator learning methods, such as DeepONet \cite{Lu2019LearningNO} and Fourier Neural Operator (FNO) \cite{li2020fourier}, focus on learning the operator that maps between the PDE parameters and their corresponding solutions. These methods are more general and do not require retraining for different PDEs, provided that the differential operator remains unchanged. 
However, the challenges of operator learning lie in handling mappings between infinite-dimensional spaces, which pose significant learning difficulties for NNs. To address these issues, we consider Green's function methods for solving PDEs in this paper as this approach does not require retraining the NN when the source function changes. 

Although there are many methods for solving PDEs via neural network structures, training remains one of the most difficult parts. The reason is that the loss function for neural networks is non-convex, making it challenging to find the global minimum. Several methods have been developed to reduce these training complexities. One approach fixes the nonlinear part and solves only the linear parameters, known as the random feature model method \cite{chen2022bridging,dong2023method,sun2024local}. Another approach involves initially finding the global minimum for a simpler neural network and then gradually smoothly increasing the network complexity. This is referred to as the greedy algorithm \cite{siegel2023greedy} or the homotopy method \cite{yang2023homotopy,zheng2024hompinns}. Another technique is to divide the neural network into several parts (stacking several simple neural networks) and train the simpler networks one by one to achieve the overall goal. Alternatively, optimizing the direction of gradient descent can improve the training dynamics, as discussed in \cite{ainsworth2022active,kingma2014adam}. 
Adam \cite{KingBa15} or SGD are frequently employed to address the resulting optimization problems in both variational and $L_2$-minimization frameworks. Furthermore, methods like Gauss-Newton and Newton type methods have been adapted to train the neural networks \cite{chen2022randomized,hao2024gauss,nltgcr}.
Lastly, using better datasets or applying known data with different levels of fidelity properly can enhance training performance \cite{howard2023multifidelity}. 
While various methods exist to improve training performance for solving PDEs, training Green's functions presents unique challenges. Their low regularity and the presence of multiple scales make it difficult to design a neural network capable of effectively learning them. Further details will be provided in the following sections of the paper.


In particular, we consider the linear PDEs in the following form:
\begin{equation}
\begin{cases}
\fL u(\vx) = g_1(\vx), & \vx \in \Omega \\
\fB u(\vx) = g_2(\vx), & \vx \in \partial \Omega
\end{cases}\label{poissongen}
\end{equation}
where $\Omega$ and $\partial \Omega$ are the domain and the domain boundary  respectively, \(g_1(\vx)\) and $g_2(\vx)$ are functions in \(\mathbb{R}\), \(u: \mathbb{R}^d \to \mathbb{R}\), $\fL$ is a linear differential operator, and $\fB$ is the operator for specifying appropriate boundary
conditions (BCs). For impulse source point \(\vy \in \Omega\), the Green's function $G(\vx, \vy)$ satisfies the following equations {(The formal definition of the Green's function for divergence elliptic equations is given in
Definition~\ref{def:green}.)}:
\begin{equation}
\begin{cases}
\fL G(\vx, \vy) = \delta(\vx - \vy), & \vx \in \Omega \\
\fB G(\vx, \vy) = 0, & \vx \in \partial \Omega
\end{cases}\label{poisson_greengen}
\end{equation}
where \(\delta(\vx)\) denotes the Dirac delta {distribution}
and both \(\fB\) and \(\fL\) operate on the \(\vx\) variable only.
Once $G(\vx, \vy)$ in \cref{poisson_greengen} is known, the weak solution $u$ of \cref{poissongen}  can be easily computed by the integration of $G$
with arbitrary forcing terms and BCs, $g_1$ and $g_2$. The formal definition of the weak solutions for divergence elliptic equations is given in Definition \ref{def:weak}.

Deriving analytical expressions for Green's functions is challenging in most cases, even for simple domains,
whereas NNs can be promising for approximating  Green's functions. Despite Green's functions having twice the dimensionality of the input domain, the same NN model can be effectively applied across different forcing terms without requiring retraining. This approach represents a specialized form of operator learning or an adaptive extension of PINNs that accommodates variations in forcing terms and BCs.
Recently, many approaches have been proposed for efficient learning of Green's functions. In \cite{gin2021deepgreen}, the authors establish a framework to learn Green's functions to solve PDEs, but they do not consider the regularity issues that bring difficulties in training. In \cite{lin2021neural}, fundamental solutions are used to address the regularity problem for learning Green's functions; however, for most PDEs, the fundamental solution is unknown. In \cite{ji2023deep}, more complex NNs, such as U-net, are used for training Green's functions. Additionally, in \cite{teng2022learning}, the domain is divided into several parts to reduce training complexity. In \cite{boulle2023learning}, the authors focus on sampling strategies to learn Green's functions rather than the structure of NNs. { \cite{wimalawarne2023learning,zhang2022mod} learn the Green's function by sampling both the source term and the corresponding solution, treating the Green operator within an operator-learning framework.  Since this entails learning in infinite-dimensional spaces, the method is computationally costly and exhibits relatively large errors.}

The primary challenge in Green's function learning approach stems from the low regularity of Green's functions due to the poor regularity of the \(\delta\) function   \cite{evans2022partial}. 
For example, the Green's function for the Poisson equation  on the unit ball in \(\mathbb{R}^d\) has the  explicit form:
\[
G(\vx, \vy)=\Phi(\vy-\vx)-\Phi(\Vert\vx\Vert_2 (\vy-\bar{\vx})),
\]
where  \(\Phi\) is the fundamental solution, \(\vx, \vy \in B_{1,\Vert\cdot\Vert_2}\), $B_{r,\Vert\cdot\Vert_2}(\vx)\subset\sR^d$ denotes a closed ball centered at $\vx\in\sR^d$ with radius $r$, measured in the Euclidean distance, and \(\bar{\vx} = \vx/\Vert \vx\Vert_2^2\).
It is straightforward to observe that $G(\vx, \vy)$ has singularities at \(\vx = \vy\) as \(\Phi\) itself is singular at \(\vzero\).

{In neural-network approximation theory \cite{mhaskar1996neural,yarotsky2017error,yarotsky2020phase,yang2023nearly,yang2024deeper}, the error estimate is typically written in the form \( \|G-G_{\vtheta}\|_{W^{s_{1},p}(\Omega\times\Omega)}=\mathcal{O}\bigl(\|G\|_{W^{s_{2},p}(\Omega\times\Omega)}\,N^{-(s_{2}-s_{1})/d}\bigr) \) for $1\le p\le\infty$, where \(N\) is the number of trainable parameters in the neural network $G_{\vtheta}$, \(s_{1}\ge 0\) specifies the norm in which the error is measured, and \(s_{2}>s_1\) is the Sobolev regularity of the Green’s function.  The exponent \(s_{2}\) depends on the coefficients and boundary conditions of the underlying PDE; see \cite{taylor2013green,hofmann2007green,kim2019green,gruter1982green} for detailed regularity results.}
When the Green's function has lower regularity, meaning that \(s_2\) is small and \(\|G\|_{W^{s_2,p}(\Omega\times\Omega)}\) is large, \(N\) must be significantly large to achieve a small approximation error.
The results in \cite{weinan2018deep,guhring2021approximation,lu2021priori,montanelli2019new,siegel2023characterization} further indicate that a large Sobolev norm of the target functions not only necessitates a large number of parameters but also increases the overall training complexity. This phenomenon has also been examined in generalization analysis, which shows that as the number of parameters increases, the complexity of networks also grows. This increased complexity demands more sample points for effective learning, thereby further escalating the overall training complexity \cite{schmidt2020nonparametric, suzuki2018adaptivity, yang2024deeper}. Thus, target functions with low regularity, such as Green's functions, necessitate a larger number of NN parameters. This, in turn, increases both the training complexity and the data requirements. Another challenge in learning Green's functions is their multiscale nature. Specifically, near $\vx=\vy$, the Green's function exhibits low regularity, while outside this region, it demonstrates a high-regularity structure \cite{bebendorf2003existence}. This disparity in regularity makes training neural networks difficult \cite{he2015delving}. When a function involves two distinct regularity scales, and the neural network's initial parameters are tuned to one scale, the network can take significantly longer to learn the low-regularity regions \cite{liu2020multi, wang2020multi, zhou2013causal}. 


A common practice in Green's function learning is to first approximate the {$\delta$ distribution} with a multidimensional Gaussian \cite{adams2003sobolev} {in the sense of distributions}, i.e.,
\begin{equation}
\label{eq: multi gauss}
\delta(\vx-\vy) \approx
\mathcal{N}_\varepsilon(\vx,\vy)=\left(\frac{1}{\varepsilon\sqrt{\pi}}\right)^d\exp\left(-\frac{\Vert\vx-\vy\Vert_2^2}{\varepsilon^2}\right),
\end{equation}
from which we approximate \cref{poisson_greengen} as
\begin{equation}
\begin{cases}
\fL G_\varepsilon(\vx, \vy)) = \mathcal{N}_\varepsilon(\vx,\vy), & \vx \in \Omega \\
\fB G_\varepsilon(\vx, \vy) = 0, & \vx \in \partial \Omega
\end{cases}\label{eq:green gen approx}
\end{equation}
where $\varepsilon$  must be carefully chosen to achieve a good approximation accuracy. 
Nevertheless, the right-hand side term $\mathcal{N}_\varepsilon(\vx,\vy)$ still exhibits poor regularity.

In this paper, we propose a \emph{multiscale} NN (MSNN) approach designed to fully exploit the multiscale structure inherent in Green's functions. For related work in the literature on  MSNN, see e.g., \cite{fan2019multiscale, liu2020multi, nah2017deep, wang2020multi}.
\emph{The basic idea in our approach is to utilize two NNs at different scales}: one to address the low regularity of $G(\vx, \vy)$ near $\vx=\vy$, while a second one approximates $G(\vx, \vy)$ across the entire domain. 
More specifically, the MSNN, denoted by $\phi(\vz, \vtheta)$ with \(\vz=(\vx,\vy)\), consists of two components:
\begin{equation}\label{nm}
    \begin{aligned}
    \phi(\vz;\vtheta)&=\frac{1}{n}\sum_{i=1}^{n}a_i\sigma_\varepsilon\left(\vw_i\cdot\vz+b_i\right)+\frac{1}{m}\sum_{i=n+1}^{n+m} a_i\sigma\left(\vw_i\cdot\vz+b_i\right) \\
    &\equiv \phi_1(\vz;\vtheta_1)+\phi_2(\vz,\vtheta_2)
    \end{aligned}
\end{equation}
where \(\vtheta=\{a_i\}_{i=1}^{m+n}\cup \{b_i\}_{i=1}^{m+n}\cup \{\vw_i\}_{i=1}^{m+n}\cup \{c\}\) 
are the parameters associated with the $m+n$ neurons and $c$ is combined from the output biases from both $\phi_1$ and $\phi_2$.
What distinguishes the two components of MSNN is the use of different activation functions, for which we have
\begin{equation}
\sigma_\varepsilon(\vx):=\varepsilon^\alpha\sigma\left(\frac{\vx}{\varepsilon^\beta}\right),
\label{eq:activation}
\end{equation} with the parameters $\alpha,\beta$ depending on the PDEs, and $\varepsilon$ is the same in \cref{eq: multi gauss}. This network structure introduces two distinct scales within the MSNN.

Following the terminology commonly used in MSNN, we refer to $\phi_1(\vz; \vtheta_1)$ as the \emph{large-scale} NN, which is designed to capture the low regularity of the Green's function near $\vx = \vy$.
On the other hand, $\phi_2(\vz; \vtheta_2)$ is the \emph{small-scale} NN, which instead focuses on learning the smooth regions of the Green's function and the residual of the first part. 
Recent studies have shown that relying solely on small-scale NNs is ineffective for accurately approximating the Green's function near $\vx=\vy$ \cite{teng2022learning}, since the NNs would require a large number of parameters
to achieve a good approximation.
Conversely, using only large-scale NN is inefficient as well, because  Green's functions exhibit high smoothness, characterized by a small high-order Sobolev norm when $\vx$ is away from $\vy$ \cite{bebendorf2003existence}. This smoothness does not necessitate large-scale networks and would lead to an unnecessary increase in the number of neurons. 
This is a phenomenon referred to as the low-rank matrix structure in \cite{bebendorf2003existence}.
We will further elaborate on these points later in the paper, using techniques from Barron space theory \cite{barron1993universal, weinan2022some, ma2020towards, siegel2022sharp, siegel2023characterization,li2024two} to demonstrate that the proposed multiscale NN structure in \cref{nm} can enhance the learning efficiency by using much fewer neurons
{with single-scale parameters} and thus significantly improve the training speed. 
Specifically, a large-scale NN is used to approximate the low-regularity region (cf. \cref{bound}), while 
a small-scale NN is subsequently employed
to approximate the residual of the first NN more efficiently (cf. \cref{bound2}). 
By ``more efficiently'', we  
mean that the  MSNN requires fewer neurons and that the magnitude of its parameters remains moderate compared to single-scale NNs. {The proposed multiscale NN shares with Fast Multipole Method (FMM) and hierarchical-matrix methods the use of a hierarchical decomposition of the Green’s function \cite{beatson1997short,GREENGARD1987325,ying2012pedestrian,smash,ddh2,hbook}. The purpose, however, is fundamentally different: FMM and hierarchical matrix methods assume the Green’s function is known and accelerates the evaluation of its associated potential. By contrast, our approach addresses the setting where the Green’s function is \emph{unknown} and must be learned directly from the governing PDE. Once learned, the approximation could be used in conjunction with FMM or other fast summation schemes for efficient evaluation.} Further details will be provided after \cref{bound2}.

Our main contributions of this work are summarized as follows:

\begin{itemize}
\item We propose an MSNN structure for learning Green's functions that effectively addresses the challenges of their low regularity. This approach enables the use of fewer parameters and eliminates the need for excessively large parameter magnitudes, making it more efficient than traditional single-scale NNs.

\item We substantiate the advantage of using MSNN through Barron space methods, introducing the concept of the multiscale Barron space method and providing approximation error bounds for Green's functions without encountering the curse of dimensionality in the \(L^\infty\)-error with second-order derivative information. 
To the best of our knowledge, this is the first work to explore the multiscale Barron space and its approximation in the context of errors with second-order derivative information.

\item  We validate our theoretical findings and methods with numerical experiments. It is evident that, with the same number of parameters, the proposed MSNN not only learns faster but also achieves better accuracy within the same training time compared to single-scale NNs.
\end{itemize}

The remaining sections are organized as follows. In \cref{preliminaries}, we provide the necessary background and preliminaries. In \cref{sec:mnn}, we introduce the proposed MSNN and prove its approximation properties. Numerical experiments are presented in \cref{num}, and we conclude in \cref{conclusion}.


\section{Preliminaries}
    \label{preliminaries}
   In this section, we introduce the notations, the Rademacher complexity and the model PDE problem that will be used in our theoretical analysis. 

\subsection{Notations}
The following notations will be used throughout the paper:
\begin{enumerate}[1.]
\item Matrices are denoted by bold uppercase letters. For example, $\vA\in\sR^{m\times n}$ is a real matrix of size $m\times n$ and $\vA^\T$ denotes the transpose of $\vA$.
\item Vectors are denoted by bold lowercase letters. For example, $\vv\in\sR^m$ is a column vector of size $m$. Furthermore, we denote by $\vv(i)$ the $i$-th element of $\vv$.
%
\item Let $B_{r,\Vert\cdot\Vert_2}(\vx)\subset\sR^d$ be the closed ball centered at $\vx$ in $\sR^d$ with radius $r$ measured by the Euclidean distance. 
%
\end{enumerate}
\subsection{Rademacher complexity}
In order to prove our main theoretical result in \cref{bound}, we first introduce the concept of the Rademacher complexity. 
\begin{defi}[{Rademacher complexity \cite{anthony1999neural}}]\label{defrad}
				Given a set of samples  $S=\{\vz_1,\vz_2,$ $\ldots,\vz_m\}$ on a domain $\fZ$, and a class $\fF$ of real-valued functions defined on $\fZ$, the empirical Rademacher complexity of $\fF$ in $S$ is defined as \[\tR_S(\fF):=\frac{1}{m}\expt_{\Xi_m}\left[\sup_{f\in\fF}\sum_{i=1}^m\xi_if(\vz_i)\right],\]where $\Xi_m:=\{\xi_1,\xi_2,\ldots,\xi_m\}$ is a set of $m$ independent random samples drawn from the Rademacher distribution, i.e., $\rmP(\xi_i=+1)=\rmP(\xi_i=-1)={1}/{2}$, for $i=1,2,\ldots,m.$ 
			\end{defi}

The following lemma bounds the expected value of the largest gap between the expected value of a function $f$ and its empirical mean by twice the expected Rademacher complexity.
			\begin{lem}[{\cite[Lemma 26.2]{shalev2014understanding}}]\label{connect1}
				Let $\fF$ be a set of functions defined on $\fZ$. Then \[\expt_{S\sim \rho^m}\sup_{f\in\fF}\left(\frac{1}{m}\sum_{i=1}^mf(\vz_i)-\expt_{\vz\sim\rho} f(\vz)\right)\le 2\expt_{S\sim\rho^m} \tR_S(\fF),\]where $S = \{\vz_1, \vz_2, \ldots, \vz_m\}$ is a set of $m$ independent random samples drawn from the distribution $\rho$.
			\end{lem}
\subsection{The Poisson Problem with Dirichlet Boundary Condition}
{We  consider the following Poisson problem
\begin{equation}
\begin{cases}
\nabla \cdot (\vA(\vx) \nabla u(\vx)) = f(\vx), & \vx \in \Omega \\
u(\vx) = 0, & \vx \in \partial \Omega
\end{cases}
\label{poisson}
\end{equation}
where \(f(\vx)\in L^\infty(\Omega)\) and \(u(\vx)\in H^1_0(\Omega)\), the closure of $C_c^\infty(\Omega)$ in $H^1(\Omega)$, and  
\(\vA=(a_{ij})_{i,j=1}^{d}\in L^\infty\!\bigl({\Omega};
\mathbb R^{d\times d}\bigr)\) is uniformly elliptic:  
\[\lambda|\boldsymbol\xi|^{2}\le \vA(\vx)\vxi\cdot \vxi:=\sum_{i,j=1}^{d}a_{ij}(\vx)\,\xi_{i}\xi_{j}\] for all \(\vx\in\Omega\) and \(\boldsymbol\xi\in\mathbb R^{d}\), with a positive constant \(\lambda\), and domain $\Omega\subset[0,1]^d$ is an open set. 
\begin{defi}\label{def:weak}
    A function \(u\in H^{1}_{0}(\Omega)\) is called a weak solution of \eqref{poisson} if  and only if
\[
  \int_{\Omega}\vA(\vx)\nabla u(\vx)\!\cdot\!\nabla\phi(\vx)\,\D\vx
  =\int_{\Omega}f(\vx)\,\phi(\vx)\,\D\vx,
  \qquad\forall\phi\in C_{c}^{\infty}(\Omega).
\]
\end{defi} By \cite[Theorem 3, p.\,301]{evans2022partial}, problem \eqref{poisson} admits a unique weak solution for every right-hand side \(f\in L^{\infty}(\Omega)\subset L^{2}(\Omega)\).

The Green's function of \eqref{poisson} is defined as follows.
\begin{defi}\label{def:green}
A function \(G:\Omega\times\Omega\to\mathbb R\) is called the
\emph{Green’s function} for problem~\eqref{poisson} if, for every
\(f\in C_c^{\infty}(\Omega)\), the representation formula
\[
  u(\vx)=\int_{\Omega} G(\vx,\vy)\,f(\vy)\,\D\vy,
  \qquad \vx\in\Omega,
\]
produces the unique weak solution \(u\) of
\eqref{poisson}.
\end{defi}

\begin{lem}[\cite{kim2019green}]\label{lem:green_exist_unique}
Assume \(\vA=(a_{ij})_{i,j=1}^{d}\in L^\infty\!\bigl({\Omega};
\mathbb R^{d\times d}\bigr)\) is uniformly elliptic and
\(\Omega\subset\mathbb R^{d}\) \emph{(}with \(d\ge 3\)\emph{)} is bounded.  
Then the Green’s function \(G\) defined in
Definition~\ref{def:green} exists and is unique. Furthermore, for any $f\in L^\infty(\Omega)$, the representation formula
\[
  u(\vx)=\int_{\Omega} G(\vx,\vy)\,f(\vy)\,\D\vy,
  \qquad \vx\in\Omega,
\]
produces the unique weak solution \(u\in H^1_0(\Omega)\) of
\eqref{poisson}.
\end{lem}
The assumption \(f \in L^\infty(\Omega)\) ensures that the integral in the representation formula is well-defined, based on the regularity of \(G(\vx,\vy)\) established in \cite[Proposition 5.3]{kim2019green}.  
Improving the regularity of the domain \(\Omega\) and the coefficient matrix \(\vA\) can lead to enhanced regularity of the Green's function \(G(\vx,\vy)\), which in turn allows the representation formula to remain well-defined for a broader class of source terms \(f\). The existence and uniqueness of the Green’s function when \(d=2\) can be found in \cite{taylor2013green}, where the notation and hypotheses are adjusted accordingly.
In this work, we focus on the most general setting; more refined regularity-based extensions will be explored in future work.

In this paper, we tend to approximate $G(\vx,\vy)$ by solving the following PDEs: For any impulse source point \(\vy \in \Omega\), \begin{equation}
    \begin{cases}
    \nabla_{\vx}\cdot(\vA(\vx) \nabla_{\vx} G_\varepsilon(\vx,\vy)) = 
    \mathcal{N}_\varepsilon(\vx,\vy), & \vx \in \Omega \\
    G_\varepsilon(\vx,\vy) = 0, & \vx \in \partial \Omega.
    \end{cases}\label{approximatedd}
\end{equation} Recall that \begin{equation}\mathcal{N}_\varepsilon(\vx,\vy)=\left(\frac{1}{\varepsilon\sqrt{\pi}}\right)^d\exp\left(-\frac{\Vert\vx-\vy\Vert_2^2}{\varepsilon^2}\right)=:\psi_\varepsilon(\vx-\vy),\label{psi}\end{equation}for a small constant $\varepsilon>0$. For every fixed \(\vy\in\Omega\) the approximate problem
\eqref{approximatedd} has a unique weak solution
\(G_{\varepsilon}(\,\cdot\,,\vy)\in H^{1}_{0}(\Omega)\) by
Lemma~\ref{lem:green_exist_unique}.
The next lemma shows that, as \(\varepsilon\to0\),
these approximate Green's functions converge to the true Green's function
\(G(\vx,\vy)\) in the weak sense.
\begin{prop}\label{prop:u_eps_convergence}
Let \(\vA=(a_{ij})_{i,j=1}^{d}\in L^{\infty}\bigl(\Omega;
\mathbb R^{d\times d}\bigr)\) be uniformly elliptic. Denote by \(G(\vx,\vy)\) the Green's function of \eqref{poisson} and by \(G_\varepsilon(\vx,\vy)\) the Green's function of the regularized problem \eqref{approximatedd}.
For any \(f\in L^{\infty}(\Omega)\) define
\[
  u(\vx)=\int_{\Omega}G(\vx,\vy)\,f(\vy)\,\D\vy,
  \qquad
  u_{\varepsilon}(\vx)=\int_{\Omega}G_{\varepsilon}(\vx,\vy)\,
       f(\vy)\,\D\vy.
\]
Then \(u_{\varepsilon}\in H^{1}_{0}(\Omega)\) and the following hold:
\begin{enumerate}\setlength{\itemsep}{4pt}
\item[\textup{(i)}]
      \(\displaystyle\lim_{\varepsilon\to0}
        \|u-u_{\varepsilon}\|_{H^{1}(\Omega)}=0.\)

\item[\textup{(ii)}]
      If in addition \(f\in H^{1}_{0}(\Omega)\), then
      \[
         \|u-u_{\varepsilon}\|_{H^{1}(\Omega)}
         \;\le\;
         C\,\varepsilon\,\|f\|_{H^{1}(\Omega)},
      \]
      where the constant \(C>0\) is independent of
      \(f\), \(u\), and \(\varepsilon\).
\end{enumerate}
\end{prop}
\begin{proof}Based on the definition of the Green's function $G$ and $\psi_\varepsilon\in L^\infty(\sR^n)$, we have that the weak solution of (\ref{approximatedd}) can be read as \begin{equation}
    G_{\varepsilon}(\vx,\vy)=\int_{\Omega}G(\vx,\vz)\psi_\varepsilon(\vz-\vy)\D \vz.
\end{equation} Then we have that \begin{equation}
    u_\varepsilon(\vx):=\int_{\Omega}G_\varepsilon(\vx,\vy)f(\vy)\D \vy=\int_{\Omega}\int_{\Omega}G(\vx,\vz)\psi_\varepsilon(\vz-\vy)f(\vy)\D \vz \D \vy.\label{Fin}
\end{equation}Then based on the Fubini's theorem, we have \[u_\varepsilon(\vx)=\int_{\Omega}G(\vx,\vz)(\psi_\varepsilon*f)(\vz)\D \vz=\int_{\Omega}G(\vx,\vy)(\psi_\varepsilon*f)(\vy)\D \vy. \]Therefore, based on $f-\psi_\varepsilon*f\in L^\infty(\Omega)$, $v=u-u_\varepsilon\in H^1_0(\Omega)$ is the weak solution of \begin{equation}\label{poissonerror}
\begin{cases}
\nabla\!\cdot\!\bigl(\vA(\vx)\nabla v(\vx)\bigr)=f-\psi_\varepsilon*f, & \vx\in\Omega,\\
v(\vx)=0, & \vx\in\partial\Omega.
\end{cases}
\end{equation}By the interior regularity of weak solutions \cite[Lemma 4.7]{kim2019green},
\[
  \|\nabla u_\varepsilon-\nabla u\|_{L^{2}(\Omega)}
  \;\le\;
  C\!\left(
      \|f-\psi_\varepsilon*f\|_{L^{2}(\Omega)}
      +\|u_\varepsilon-u\|_{L^{2}(\Omega)}
   \right),
\]
and by the \(L^{2}\)-bounded inverse estimate
\cite[Theorem 6, p.\,306]{evans2022partial},
\[
  \|u_\varepsilon-u\|_{L^{2}(\Omega)}
  \;\le\;
  C\,\|f-\psi_\varepsilon*f\|_{L^{2}(\Omega)} .
\]Putting the two bounds together we obtain
\begin{equation}\label{eq:H1_error}
  \|u_\varepsilon-u\|_{H^{1}(\Omega)}
  \;\le\;
  C\,\|f-\psi_\varepsilon*f\|_{L^{2}(\Omega)},\footnote{%
  The constant \(C\) may change from line to line.}
\end{equation}
where \(C>0\) is independent of \(f\), \(u\), and \(\varepsilon\). As for \(\|f-\psi_\varepsilon*f\|_{L^{2}(\Omega)}\) with $f\in L^2(\Omega)$, let \(h\in C^{\infty}_{c}(\mathbb{R}^{d})\) be a non-negative bump function supported in \([-1,1]^{d}\) with
\(h(\vx)=1\) on \(\bigl[-\tfrac12,\tfrac12\bigr]^{d}\) and \(0\le h\le1\) everywhere.
Define \(g_{\varepsilon}(\vx)=h(\vx/\sqrt{\varepsilon})\,\psi_{\varepsilon}(\vx)\), then
\begin{equation}
  \|f-\psi_{\varepsilon}*f\|_{L^{2}(\Omega)}
  \;\le\;
  \underbrace{\Bigl\|f-\frac{g_{\varepsilon}*f}{\|g_{\varepsilon}\|_{L^{1}(\sR^d)}}\Bigr\|_{L^{2}(\Omega)}}_{R_1}
  +\underbrace{\Bigl\|\Bigl(\frac{g_{\varepsilon}}{\|g_{\varepsilon}\|_{L^{1}(\sR^d)}}-\psi_{\varepsilon}\Bigr)*f\Bigr\|_{L^{2}(\Omega)}}_{R_2}.\label{mollifier}
\end{equation}

The term \(R_{1}\) vanishes as \(\sqrt{\varepsilon}\to0\) because the kernel
\(g_{\varepsilon}/\|g_{\varepsilon}\|_{L^{1}(\mathbb{R}^{d})}\)
is a compactly supported \((-\sqrt{\varepsilon},\sqrt{\varepsilon})\)-mollifier satisfying
\(\int_{\mathbb{R}^{d}} g_{\varepsilon}(\vx)/\|g_{\varepsilon}\|_{L^{1}}\D \vx = 1\)
for every \(\varepsilon>0\); hence
\cite[Theorem~2.29]{adams2003sobolev} applies.  For the second term, we estimate
\[
\begin{aligned}
  R_2
 &\le
  \Bigl(
    \|g_{\varepsilon}-\psi_{\varepsilon}\|_{L^{1}(\sR^d)}
    +\Bigl|\frac{1}{\|g_{\varepsilon}\|_{L^{1}(\sR^d)}}-1\Bigr|
      \,\|\psi_{\varepsilon}\|_{L^{1}(\sR^d)}
  \Bigr)\,\|f\|_{L^{2}(\Omega)} \\
 &\le
  \Bigl(
    \|\psi_{\varepsilon}\|_{L^{1}\bigl(\mathbb{R}^{d}\backslash\left[-\tfrac{\sqrt{\varepsilon}}{2},\tfrac{\sqrt{\varepsilon}}{2}\right]^{d}\bigr)}
    +\Bigl|\frac{1}{\|g_{\varepsilon}\|_{L^{1}(\sR^d)}}-1\Bigr|
  \Bigr)\,\|f\|_{L^{2}(\Omega)}.
\end{aligned}
\]

Since \(\psi\) is a Gaussian mollifier, it satisfies
\[
  \lim_{\varepsilon\to0}\|\psi_{\varepsilon}\|_{L^{1}\!\left(\left[-\tfrac{\sqrt{\varepsilon}}{2},\tfrac{\sqrt{\varepsilon}}{2}\right]^{d}\right)}
  \;=\;
  1
  \quad\text{and}\quad
  \lim_{\varepsilon\to 0}\|g_{\varepsilon}\|_{L^{1}(\mathbb R^{d})}=\|\psi_{\varepsilon}\|_{L^{1}(\mathbb R^{d})}=1,
\]
so the two \(L^{1}\)-differences in parentheses decay to $0$ as
\(\varepsilon\to0\).
Hence \(\|f-\psi_{\varepsilon}*f\|_{L^{2}(\Omega)}\to0\), which completes the proof of~(i). 

If, moreover, \(f\in H^{1}_{0}(\Omega)\), extend \(f\) to
\(\widetilde{f}\in H^{1}(\mathbb{R}^{d})\) by setting
\(\widetilde{f}=f\) on \(\Omega\) and \(\widetilde{f}=0\) on
\(\mathbb{R}^{d}\setminus\Omega\).  Then
\[
  \|f-\psi_{\varepsilon}*f\|_{L^{2}(\Omega)}
  =\|\widetilde{f}-\psi_{\varepsilon}*\widetilde{f}\|_{L^{2}(\Omega)}
  \le \|\widetilde{f}-\psi_{\varepsilon}*\widetilde{f}\|_{L^{2}(\mathbb{R}^{d})}.
\]
Applying \cite[Lemma 3.5 (iv), p.\,98]{majda2002vorticity} (see
Remark~\ref{proof:pro} for details) gives
\begin{equation}
  \|\widetilde{f}-\psi_{\varepsilon}*\widetilde{f}\|_{L^{2}(\mathbb{R}^{d})}
  \;\le\;
  C\,\varepsilon\,\|\widetilde{f}\|_{H^{1}(\mathbb{R}^{d})}
  =C\,\varepsilon\,\|f\|_{H^{1}(\Omega)},
 \label{(ii)proof}   
\end{equation}
which proves part~(ii).
\end{proof}
\begin{rmk}\label{proof:pro}
\begin{enumerate}\setlength{\itemsep}{4pt}
\item
In \eqref{(ii)proof} we invoke \cite[Lemma 3.5 (iv), p.\,98]{majda2002vorticity}.  
That lemma does \emph{not} require the mollifier to be compactly supported; it only needs  
\(\int_{\mathbb R^{d}}\psi=1\), \(|\widehat{\psi}(\boldsymbol\xi)-1|\le C|\boldsymbol\xi|^{2}\), and  
\(\widehat{\psi}(\boldsymbol\xi)\le C\).  
These conditions are satisfied by the Gaussian kernel used here.

\item
If one assumes only \(f \in H^1(\Omega)\) (without the zero boundary condition), a similar convergence result may still be obtained by applying an extension operator for Lipschitz domains; see \cite{adams2003sobolev}.  
However, an additional error term appears in the estimate, since near the boundary \(\psi_{\varepsilon} * f \neq \psi_{\varepsilon} * \widetilde{f}\), leading to a boundary discrepancy in \(\|f - \psi_{\varepsilon} * f\|_{L^2(\Omega)}\).  
Moreover, if \(f \in H^k(\Omega)\), the ideal convergence rate should be faster than $\fO(\varepsilon)$.  
A more detailed analysis in this setting will be considered in future work.
\end{enumerate}
\end{rmk}
Next we approximate the regularized Green's function \(G_{\varepsilon}\) by a
smooth neural network \(G_{\boldsymbol\theta}\) and then quantify the
resulting solution error.
\begin{prop}\label{prop:NN_vs_Geps}
    Assume \(\vA\in W^{1,\infty}\!\bigl({\Omega};\mathbb R^{d\times d}\bigr)\)
is uniformly elliptic.  
Let \(G_{\varepsilon}(\vx,\vy)\) be the Green's function of the mollified
problem \eqref{approximatedd}.
Suppose a neural network \(G_{\boldsymbol\theta}(\cdot,\vy)\in
W^{2,\infty}(\Omega)\cap H^1_0(\Omega)\) with \(\sup_{\vy\in\Omega}\|G(\cdot,\vy)\|_{H^1(\Omega)}<\infty\) satisfies
\begin{equation}\label{eq:small_residual}
  \sup_{\vy\in\Omega}
  \Bigl\|
      \mathcal N_{\varepsilon}(\vx,\vy) \;-\;
      \nabla_{\!x}\!\cdot\!\bigl(\vA(\vx)\nabla_{\!x}
        G_{\boldsymbol\theta}(\vx,\vy)\bigr)
  \Bigr\|_{L^2(\Omega)}
  \;\le\;\delta,
\end{equation}
for some \(\delta>0\).
Define
\[
  u_{\varepsilon}(\vx)=\int_{\Omega}G_{\varepsilon}(\vx,\vy)\,f(\vy)\,\D \vy,
  \quad
  u_{\boldsymbol\theta}(\vx)=
      \int_{\Omega}G_{\boldsymbol\theta}(\vx,\vy)\,f(\vy)\,\D \vy,
  \qquad f\in L^{1}(\Omega).
\]
Then \(u_\varepsilon,u_{\boldsymbol\theta}\in H^{1}(\Omega)\) with $u_{\boldsymbol{\vtheta}}$ satisfying the homogeneous Dirichlet boundary condition, and
\[
  \|u_{\varepsilon}-u_{\boldsymbol\theta}\|_{H^{1}(\Omega)}
  \;\le\;C\,\delta\,\|f\|_{L^{1}(\Omega)},
\]
where \(C>0\) depends only on \(\Omega\) and the ellipticity bounds of
\(\vA\), and is independent of \(f\), \(\varepsilon\), and \(\delta\).
\end{prop}

\begin{proof}
First, we define \( u_\varepsilon = \int_{\Omega} G(\vx,\vy)\, (f * \psi_\varepsilon)(\vy)\, \D \vy \) based on \eqref{Fin}.  
Then, by Lemma~\ref{lem:green_exist_unique} and the fact that \( f * \psi_\varepsilon \in L^\infty(\Omega) \), it follows that \( u_\varepsilon \in H^1_0(\Omega) \).  
Furthermore, we estimate
\[
\|u_{\boldsymbol{\theta}}\|_{H^1(\Omega)} \le \int_\Omega \|G(\cdot,\vy)\|_{H^1(\Omega)}\, |f(\vy)|\, \D \vy 
\le \sup_{\vy \in \Omega} \|G(\cdot,\vy)\|_{H^1(\Omega)} \cdot \|f\|_{L^1(\Omega)} < \infty,
\]
and \( u_{\boldsymbol\theta}(\vx_*) = \int_{\Omega} G_{\boldsymbol\theta}(\vx_*,\vy)\, f(\vy)\, \D \vy = 0 \) for all \( \vx_* \in \partial\Omega \).

For fixed \(\vy \in \Omega\), since \(\vA \in W^{1,\infty}(\Omega;\mathbb{R}^{d \times d})\) and \(G_{\boldsymbol\theta}(\cdot,\vy) \in W^{2,\infty}(\Omega) \cap H^1_0(\Omega)\), it follows that
\[
q_{\vy}(\vx) := \nabla_{\!x} \cdot \bigl( \vA(\vx)\nabla_{\!x} G_{\boldsymbol\theta}(\vx,\vy) \bigr) \in L^\infty(\Omega).
\]
Since \(G_{\boldsymbol\theta}(\cdot,\vy)\) (respectively, \(G_\varepsilon(\cdot,\vy)\)) can be regarded as the solution of problem~\eqref{poisson} with source term \(q_{\vy}(\vx)\) (respectively, \(\mathcal{N}_\varepsilon(\cdot,\vy)\)) for any \(\vy \in \Omega\), it follows that the difference \(G_{\varepsilon}(\cdot,\vy) - G_{\boldsymbol\theta}(\cdot,\vy)\) is the solution of problem~\eqref{poisson} with source term \(\mathcal{N}_\varepsilon(\cdot,\vy) - q_{\vy}\).

Furthermore, set \(w = u_{\varepsilon} - u_{\boldsymbol\theta} = \int_\Omega \big(G_{\varepsilon}(\vx,\vy) - G_{\boldsymbol\theta}(\vx,\vy)\big) f(\vy)\, \D\vy\). Based on Minkowski’s integral inequality, we have
\begin{align}
  \|u_{\varepsilon} - u_{\boldsymbol\theta}\|_{H^1(\Omega)}
  &=
  \left\|\int_\Omega \big(G_{\varepsilon}(\cdot,\vy) - G_{\boldsymbol\theta}(\cdot,\vy)\big) f(\vy)\, \D\vy\right\|_{H^1(\Omega)} \notag\\
  &\le
  \int_\Omega \|G_{\varepsilon}(\cdot,\vy) - G_{\boldsymbol\theta}(\cdot,\vy)\|_{H^1(\Omega)} |f(\vy)|\, \D\vy \notag\\
  &\le
  \sup_{\vy \in \Omega} \|G_{\varepsilon}(\cdot,\vy) - G_{\boldsymbol\theta}(\cdot,\vy)\|_{H^1(\Omega)} \cdot \|f\|_{L^1(\Omega)} \notag\\
  &\le
  C \sup_{\vy\in\Omega}
      \| \mathcal{N}_{\varepsilon}(\cdot,\vy)
    -q_{\vy}\|_{L^{2}(\Omega)}\cdot
  \|f\|_{L^{1}(\Omega)} 
  \;\le\; C\delta \,
  \|f\|_{L^{1}(\Omega)},\notag
\end{align}
where the third inequality uses the \(H^1\)-regularity estimate for the weak solution (cf.~\eqref{eq:H1_error}) together with the assumption that \(G_{\boldsymbol\theta}\) satisfies the homogeneous Dirichlet boundary condition.
\end{proof}

\begin{rmk}\label{rmk:G_NN}
The assumptions \(G_{\boldsymbol\theta}(\cdot,\vy)\in
W^{2,\infty}(\Omega)\) and the homogeneous Dirichlet boundary condition
can be met in practice:  
one uses a smooth activation (e.g.\ \(\tanh\)) to guarantee
\(W^{2,\infty}\)-regularity, and enforces the boundary condition either by
a structure‐preserving network construction
\cite{gu2021structure,luo2020two} or by adding a large penalty term in the
loss.  The remaining task is therefore to obtain
\(G_{\boldsymbol\theta}\) so that the residual bound
\eqref{eq:small_residual} is sufficiently small. In this paper we accomplish that with a multiscale neural-network design.
\end{rmk}}

\section{Multiscale Neural Networks (MSNNs)}
\label{sec:mnn}

In this section, we provide details on the construction of the MSNN in \cref{nm}. In particular, we will analyze the selection criteria for activation functions $\sigma$ and $\sigma_\varepsilon$ in \cref{eq:activation} and prove how these choices contribute to the approximation properties of NNs. 

\subsection{Multiscale Barron space and activation functions}
We assume that  the following assumptions hold for the activation functions $\sigma$ and matrix $\vA(\vx)$:
\begin{assump}\label{assump2}
For arbitrary $\vv=(\vv_{\vx},\vv_{\vy})\in\sR^{2d}$ and $\vz_i=(\vx_i,\vy_i)\in[0,1]^{2d}$ for $i=1,2$, 
let $\gamma_i= \vv \cdot \vz_i + b$ and $\beta=(\|\vv\|_1+|b|)^2$.
There exists a constant $M\ge 0$ such that   $\sigma$ and $\vA(\vx)$  satisfy:
\begin{align}
    \abs{\sigma''\left(\gamma_1\right) \vv_{\vx}^\top \vA(\vx_1)\vv_{\vx}  - \sigma''\left(\gamma_2\right) \vv_{\vx}^\top \vA(\vx_2)\vv_{\vx}}
    &\leq \beta M|{\vv\cdot(\vz_1-\vz_2)}|,  \\
    \abs{\sigma'\left(\gamma_1\right) (\nabla\cdot\vA(\vx_1))^\top\vv_{\vx} - 
          \sigma'\left(\gamma_2\right) (\nabla\cdot\vA(\vx_2))^\top\vv_{\vx}} &\le 
           \beta M \abs{\vv\cdot (\vz_1-\vz_2)}.
\end{align}
   \end{assump}

   \begin{rmk}
     When \(\vA(\vx)\) is a constant matrix, the assumption can be directly obtained from the Lipschitz property of the activation function: \( |\sigma^{(k)}(x)| \leq \sqrt{M}|x|^{3-k} \) for \(k=2,3\), and the boundedness of \(\vA\), i.e., \( |\vv_{\vx}^\top \vA \vv_{\vx}| \leq \sqrt{M}\|\vv_{\vx}\|_1^2 \). This condition holds for most of the common activation functions, such as 
     \(\arctan(x)\) and \(\tanh(x)\) used in the experiments in \cref{num}. In the case of variable \(\vA(\vx)\), the additional requirement is that \(\vA(\vx)\) satisfies a special type of Lipschitz property, specifically in the \(\vv\)-direction when combined with \(\vv_{\vx}\) and \(\sigma\). However, this is not the standard definition of Lipschitz continuity, so we will not delve into it further here and will simply make an assumption at this point.

   \end{rmk}
%
Next, we describe the methodology for choosing the activation functions
in terms of the parameters \(\alpha\) and \(\beta\) of $\sigma_\varepsilon$ in \cref{eq:activation}. The choice of \(\alpha\) and \(\beta\) should effectively reflect the structure of the Green's function, i.e., thereby enabling NNs to approximate the Green's function accurately while minimizing the number of parameters. The optimal values of \(\alpha\) and \(\beta\) depend on each other. For simplification, we  set \(\beta = 1\) and focus on determining the appropriate \(\alpha\). This approach of choosing \(\alpha\) can be applied to other values of  \(\beta\). 
The selection of \(\alpha\) relies on the Barron space method \cite{barron1993universal,weinan2022some,lu2021priori,siegel2022sharp,siegel2023characterization}, which is a technique for approximating specific target functions using shallow NNs without suffering from \emph{the curse of dimensionality}. Given the multiscale nature of the proposed approach, we adapt this technique into a  \emph{multiscale Barron space method}. This adaptation is necessary to evaluate errors using \(L^\infty\)-error with second-order derivative information, particularly relevant for solving second-order PDEs. To the best of our knowledge, this is the first attempt to apply the Barron space method in such a context.

When \(\alpha^* \in\sR\) and \(\beta = 1\) in \cref{eq:activation}, the large-scale NN $\phi_1$ in \cref{nm} can be simplified as
\begin{align}
\phi_1(\vz;\vtheta_1)&=\frac{\varepsilon^{\alpha^*} }{n}\sum_{i=1}^{n}a_i\sigma\left(\vw_i\cdot{\vz}/{\varepsilon}+{b_i}/{\varepsilon}\right), \quad \vz=(\vx, \vy), \quad \vw_i=(\vw_{i,\vx},\vw_{i,\vy}), 
\label{eq:ph11}
\end{align}
applying the linear differential operator in \cref{approximatedd}, it follows that
\begin{align} \label{eq:diffphi0}
\nabla\cdot(\vA(\vx) \nabla \phi_1) &= \frac{\varepsilon^{-2+\alpha^*}}{n} \sum_{i=1}^{n}\zeta(\vz; a_i, \vw_i, b_i),
\end{align}
where
\begin{align} \label{eq:diffphi1}
\zeta(\vz; a_i, \vw_i, b_i) &= a_i \left[ \vw_{i,\vx}^\top \vA(\vx)\vw_{i,\vx}\sigma''\left(f_i\right) +  
{\varepsilon}(\nabla\cdot\vA(\vx))^\top\vw_{i,\vx}\sigma'\left(f_i\right) \right],
\end{align}
where $f_i = (\vw_i\cdot{\vz} + {b_i})/{\varepsilon}$.
This leads to the following definition of the so-called $\varepsilon,\alpha^*$-$\vA(\vx)$-Barron function, which  reduces to the classical Barron space as discussed in \cite{weinan2022some,luo2020two} when \(\varepsilon = 1\).


\begin{defi}\label{barronde}[$\varepsilon,\alpha^*$-$\vA(\vx)$-Barron function]
For any $\varepsilon\in(0,1]$ and $\alpha^*\in \sR$, function $f(\vx,\vy):\Omega\times\Omega\to \sR$ is called an $\varepsilon,\alpha^*$-$\vA(\vx)$-Barron function with respect to $\sigma(x)$, if 
there exists a probability distribution $\rho$ over $\sR^{d+2}$ such that 
\begin{align}
f(\vx,\vy) = \expt_{(a,\vw,b)\sim\rho} \varepsilon^{-2+\alpha^*} \zeta(\vz; a, \vw, b).
\label{barronpre}
\end{align}
\end{defi}
For $\varepsilon,\alpha^*$-$\vA(\vx)$-Barron functions, we can define the Barron-norm constant.
\begin{defi}\label{constant}
Suppose that the function $f(\vx,\vy):\Omega\times\Omega\to \sR$ is an $\varepsilon,\alpha^*$-$\vA(\vx)$-Barron function with respect to $\sigma(x)$. The Barron-norm constant of $f$ is defined as
\begin{equation}
    \|f\|_{\fB_{\varepsilon,\alpha^*}} := \inf_{\rho \in \mathcal{P}} \sqrt{\expt_{(a, \vw, b) \sim \rho} (|a| \left(\|\vw\|_1^3 + 2\|\vw\|_1^2 |b| + \|\vw\|_1 |b|^2\right) )^2}\label{barronconstant}
\end{equation}
where $\mathcal{P}$ is the set of all probability distributions $\rho$ over $\mathbb{R}^{d+2}$ that satisfy  \cref{barronpre}.
\end{defi}
This definition is consistent with that in \cite{ma2020towards,ma2022barron} when \(p=2\) in their context and  
match the definition in \cite{luo2020two} when \((\vw,b)\) is replaced by \(\bar{\vw}\). In \cite{ma2020towards,ma2022barron}, the authors use the definition of the Barron norm for \(p=\infty\) because the space is sufficiently large when the activation functions are ReLU due to the homogeneity of ReLU, which ensures that the spaces for \(1 \le p \le \infty\) are equivalent. However, since our work does not restrict to ReLU activation function, the proof differs, and we perform it specially for the \(p=2\) case. 

The difference between our definition of Barron-type space and others is that we introduce a smaller term \(\varepsilon\) in the definition, which we call the \emph{multiscale Barron space method}. The following definitions are also based on the multiscale Barron space method. We first establish a large-scale NN to approximate the Green's function, in \cref{bound}, based on the large-scale Barron space, i.e., \(\varepsilon,\alpha^*\)-\(\vA(\vx)\)-Barron function for \(\varepsilon \ll 1\). Then, we establish a small-scale NN based on the small-scale Barron space, i.e., \(\varepsilon=1\), in \cref{barronde}, to approximate the residual from the first step.

\begin{lem}\label{twobarron}
    For $\alpha^*_1,\alpha^*_2\in\sR$ and $\varepsilon\in(0,1]$, if function $f$ is an $\varepsilon,\alpha^*_1$-$\vA(\vx)$-Barron function, then it is also an $\varepsilon,\alpha^*_2$-$\vA(\vx)$-Barron function. Furthermore, we have 
\begin{align} \label{eq:ff}
    \frac{\|f\|_{\mathcal{B}_{\varepsilon,\alpha^*_1}}}{\|f\|_{\mathcal{B}_{\varepsilon,\alpha^*_2}}} =\varepsilon^{-\alpha^*_1+\alpha^*_2}.
\end{align}
\end{lem}
\begin{proof}
  For \(f\) to be an \(\varepsilon, \alpha^*_1\)-\(\vA(\vx)\)-Barron function, it means that we can find a probability distribution \(\rho_1\) such that 
\begin{align}
f(\vx,\vy) = \expt_{(a,\vw,b)\sim\rho_1} \varepsilon^{-2+\alpha^*_1} a \left[ \vw_{\vx}^\top \vA(\vx)\vw_{\vx}\sigma''\left(d\right) + \varepsilon(\nabla\cdot\vA(\vx))^\top\vw_{\vx}\sigma'\left(d\right) \right],
\end{align}
where $d = {(\vw\cdot{\vz} + {b})}/{\varepsilon}$. With simple calculation, we have
\begin{align}
f(\vx,\vy) = \expt_{(a,\vw,b)\sim\rho_1} \varepsilon^{-2+\alpha^*_2}\varepsilon^{\alpha^*_1-\alpha^*_2} a \left[ \vw_{\vx}^\top \vA(\vx)\vw_{\vx}\sigma''\left(d\right) + \varepsilon(\nabla\cdot\vA(\vx))^\top\vw_{\vx}\sigma'\left(d\right) \right].
\end{align}
Then set \(\widehat{a} = \varepsilon^{\alpha^*_1-\alpha^*_2} a\),  we can show that there exists a distribution \(\rho_2\) such that 
\begin{align}
f(\vx,\vy) = \expt_{(\widehat{a},\vw,b)\sim\rho_2} \varepsilon^{-2+\alpha^*_2} \widehat{a} \left[ \vw_{\vx}^\top \vA(\vx){\vw_{\vx}} \sigma''\left(d\right) + \varepsilon(\nabla\cdot\vA(\vx))^\top{\vw_{\vx}}\sigma'\left(d\right) \right],
\end{align}
where \(\rho_2\) is just a rescaling of \(\rho_1\). 
Furthermore, we have 
\begin{align}
\expt_{(a, \vw, b) \sim \rho_1} (\eta|a|)^2 = \varepsilon^{-2(\alpha^*_1-\alpha^*_2)}\expt_{(\widehat{a},{\vw},{b})\sim\rho_2} (\eta|\widehat{a}|)^2,
\end{align}
where $\eta=\|\vw\|_1^3 + 2\|\vw\|_1^2 |b| + \|\vw\|_1 |b|^2$.
The above still holds for the \(\inf_{\rho \in \mathcal{P}}\) calculation since they are proportional, therefore 
\cref{eq:ff} follows.
\end{proof}

\begin{rmk}
    \cref{twobarron} tells us two facts: First, in a {multiscale Barron space method}, for a fixed $\varepsilon$, whether a function is an 
    \(\varepsilon,\alpha^*\)-\(\vA(\vx)\)-Barron function or not does not depend on \(\alpha^*\), although 
    the magnitude of the Barron-norm constant does. This shows that allowing a function to belong to an \(\varepsilon,\alpha^*\)-\(\vA(\vx)\)-Barron function for all \(\alpha^*\) for a fixed $\varepsilon$ is not an overly restrictive assumption. In the following sections, we will perform the theoretical analysis based on the assumption that for a fixed $\varepsilon$, \(\mathcal{N}_\varepsilon(\vx,\vy)\) is an \(\varepsilon,\alpha^*\)-\(\vA(\vx)\)-Barron function for all \(\alpha^{*}\). Given the high smoothness of \(\mathcal{N}_\varepsilon(\vx,\vy)\), this assumption should hold in most cases based on \cite[Theorem 3.1]{wojtowytsch2022representation}. If not, we can still establish the approximation theory using Sobolev approximation \cite{siegel2022optimal,yang2023nearly}, which, despite the curse of dimensionality, retains the advantage of the MSNN. Second, for a fixed 
    \(\varepsilon \in (0,1)\), a larger \(\alpha^*\) results in a larger Barron-norm constant. 
    When \(\varepsilon = 1\), the constants are the same, thus we denote 
    by $\|f\|_{\mathcal{B}_{1}}$ the Barron-norm constant and refer to $1,\alpha^*$-$\vA(\vx)$-Barron functions as $1$-$\vA(\vx)$-Barron functions  for simplicity.
\end{rmk}


\subsection{Approximation property of MSNN}
Under \cref{assump2}, we can show the approximation rate of the proposed MSNN in the next two theorems.
The difficulties in our approximation arise because we are approximating the Green's function based on the PDE~\cref{approximatedd}, which requires the approximation to contain derivative information for both the NNs and the Green's function. 
 Classical \(L^2\) or \(L^\infty\)-approximations, which do not contain derivative information, such as those in \cite{ma2020towards}, are not sufficient. Furthermore, when approximating Green's functions, the most challenging part is the region near \(\vx = \vy\). If we consider \(H^2\) approximation, we cannot ensure that this region is approximated well when the overall error is small, because the measure of that region is small compared to the whole domain. Therefore, we need to consider the \(L^\infty\)-error with second-order derivative information in the approximation. { Finally, we obtain Barron-space approximation bounds for activation functions that do not share the homogeneity property of \(\mathrm{ReLU}^{k}\); because such homogeneity would destroy the multiscale structure, our proof differs from the methods in \cite{luo2020two,gu2023stationary}.
 } 

 \begin{thm}\label{bound}
       For any $\varepsilon\in(0,1]$, suppose  
       $\mathcal{N}_\varepsilon(\vx,\vy):\Omega\times\Omega\to \mathbb{R}$ is an $\varepsilon,\alpha^*$-$\vA(\vx)$-Barron function with respect to $\sigma(x)$ for all $\alpha^*\in \sR$, and \cref{assump2} holds
       and $\vz = (\vx, \vy)$.
       Then for any $n \ge 1$, there exists a set of large-scale NNs, denoted by $\mathcal{A}_n$, with activation function $\sigma_\varepsilon$ in \cref{eq:activation}, a finite $\alpha^*$ associated with $\beta=1$, i.e.,
\begin{equation}
\mathcal{A}_n \subset \left\{\phi_1 \mid \phi_1(\vz; \vtheta) =  \frac{1}{n}\sum_{i=1}^n a_i \sigma_\varepsilon\left( \vw_i \cdot \vz + b_i\right)
\right\}
\end{equation}
such that for any $\phi_1(\vz; \vtheta) \in \mathcal{A}_n$, we have
\begin{equation}
\sup_{\vx,\vy} \abs{\nabla \cdot \left( \vA(\vx) \nabla \phi_1(\vz; \vtheta) \right) - \mathcal{N}_\varepsilon(\vx,\vy)} 
\le \frac{8M\sqrt{2d}\|\fN_\varepsilon\|_{\fB_{\varepsilon,\alpha^*}}}{\varepsilon^{-\alpha^*+3}\sqrt{n}} 
\label{first_part}
\end{equation}
and 
\begin{equation}
\frac{1}{n} \sum_{i=1}^n \left(|a_i| \left( \|\vw_i\|_1^3 + 2\|\vw_i\|_1^2 |b_i| + \|\vw_i\|_1 |b_i|^2\right) \right)^2\le 3\|\fN_\varepsilon\|^2_{\fB_{\varepsilon,\alpha^*}}\label{order1}
\end{equation}
where $M$ is the  constant defined in \cref{assump2}. 
\end{thm}
\begin{proof}
By the definitions of $\varepsilon,\alpha^*$-$\vA(\vx)$-Barron function, there exists a probability density \(\rho\) such that 
\begin{align}
\fN_\varepsilon(\vx,\vy) &= \expt_{(a,\vw,b)\sim\rho} {\varepsilon^{-2+\alpha^*}} \zeta(\vz; a, \vw, b),
\label{barronpre1}
\end{align}
for all \(\vz=(\vx,\vy) \in \Omega \times \Omega\), 
where $\zeta$ is defined in \cref{eq:diffphi1},
and  




\begin{equation} \label{eq:less1}
\sqrt{\expt_{(a, \vw, b) \sim \rho} \left[|a| \left(\|\vw\|_1^3 + 2\|\vw\|_1^2 |b| + \|\vw\|_1 |b|^2\right)\right]^2} \le (1+\delta)\|\fN_\varepsilon\|_{\fB_{\varepsilon,\alpha^*}}<\frac{9}{8}\|\fN_\varepsilon\|_{\fB_{\varepsilon,\alpha^*}},
\end{equation}where $\delta$ is a small constant larger than $0$.
We need to prove there exists $\{\zeta_i\}_{i=1}^n$ such that
\begin{equation}
\sup_{\vz} \left|\frac{1}{n} \sum_{i=1}^{n} \left({\varepsilon^{-2+\alpha^*}}\zeta_i - \fN_\varepsilon\right)\right| \le \frac{8 M\sqrt{2d}\|\fN_\varepsilon\|_{\fB_{\varepsilon,\alpha^*}}}{\varepsilon^{-\alpha^*+3}\sqrt{n}},
\quad \zeta_i \equiv \zeta(\vz; a_i, \vw_i, b_i).
\end{equation} The idea of the proof is to choose $\{\zeta_i\}_{i=1}^n$ 
with i.i.d. probability distribution $(a_i,\vw_i,$ $b_i)_{i=1}^n \sim \rho^n\) for \(i = 1, \ldots, n$,
then bound the expectation of \(\frac{1}{n} \sum_{i=1}^{n} \left({\varepsilon^{-2+\alpha^*}}\zeta_i - \mathcal{N}_\varepsilon\right)\). 
This shows that a sample exists that satisfies the error bound established earlier.


From \cref{connect1}, we have 
\begin{equation}\label{rad}
\begin{aligned}
&\expt_{(a_i,\vw_i,b_i)_{i=1}^n\sim \rho^n}\left [\sup_{\vz}\frac{1}{n}\sum_{i=1}^{n}\left({\varepsilon^{-2+\alpha^*}}\zeta_i - \fN_\varepsilon\right)\right] \\
&\le \frac{2}{\varepsilon^{2-\alpha^*}}\expt_{(a_i,\vw_i,b_i)_{i=1}^n\sim\rho^n}\expt_{\Xi_n}\left[\sup_{\vz}\frac{1}{n}\sum_{i=1}^n\xi_i \zeta_i\right]
\end{aligned}
\end{equation}
where $\Xi_n=\{\xi_1,\xi_2,\ldots,\xi_n\}$ are independent random samples drawn from the Radema- cher distribution, i.e., 
$\rmP(\xi_i=+1)=\rmP(\xi_i=-1)={1}/{2}$, for $i=1,2,\ldots,n$. The inequality \cref{rad} is by the definition of the Rademacher complexities in \cref{defrad}.

We next estimate the last term of \cref{rad}, i.e.,
\begin{align}
\expt_{\Xi_n}\left[\sup_{\vz}\frac{1}{n}\sum \xi_i \zeta_i\right] \le ~ &\expt_{\Xi_n}\left[\sup_{\vz}\frac{1}{n}\sum a_i \xi_i\vw_{i,\vx}^\top \vA(\vx)\vw_{i,\vx}  
\sigma''\left(f_i(\vz)\right)\right]
+ \\
&\expt_{\Xi_n}\left[\sup_{\vz}\frac{1}{n}\sum a_i \xi_i \varepsilon (\nabla\cdot\vA(\vx))^\top \vw_i \sigma'\left(f_i(\vz)\right)\right],
\end{align}
where $f_i(\vz) = (\vw_i\cdot {\vz}+{b_i})/{\varepsilon}$.
%
%
For the first term of the right-hand side of the above inequality, we have \begin{align}
   & \expt_{\Xi_n}\left[\sup_{\vz}\frac{1}{n}\sum \xi_i a_i \vw_{i,\vx}^\top \vA(\vx)\vw_{i,\vx} \sigma''\left(f_i(\vz)\right)\right] \notag\\=~&\frac{1}{2n}\expt_{\xi_2,\ldots,\xi_m}\Bigg[\sup_{\vz}\Bigg(\ a_1 \vw_{1,\vx}^\top \vA(\vx)\vw_{1,\vx} \sigma''\left(f_1(\vz)\right)\notag\\
   &+\sum_{i=2}^m\xi_i a_i \vw_{i,\vx}^\top \vA(\vx)\vw_{i,\vx} \sigma''\left(f_i(\vz)\right)\Bigg)\notag\\
   &+\sup_{\vz}\left(- a_1 \vw_{1,\vx}^\top \vA(\vx)\vw_{1,\vx} \sigma''\left(f_1(\vz)\right)+\sum_{i=2}^m\xi_i a_i \vw_{i,\vx}^\top \vA(\vx)\vw_{i,\vx} \sigma''\left(f_i(\vz)\right)\right)\Bigg]\notag\\
   =~&\frac{1}{2n}\expt_{\xi_2,\ldots,\xi_m}\Bigg[\sup_{\vz_1,\vz_2}\Bigg(\ a_1 \vw_{1,\vx}^\top \vA(\vx_1)\vw_{1,\vx} \sigma''\left(f_1(\vz_1)\right)- a_1 \vw_{1,\vx}^\top \vA(\vx_2)\vw_{1,\vx} \sigma''\left(f_1(\vz_2)\right)\notag\\&+\sum_{i=2}^m\xi_i a_i \vw_{i,\vx}^\top \vA(\vx_1)\vw_{i,\vx} \sigma''\left(f_i(\vz_1)\right)+\sum_{i=2}^m\xi_i a_i \vw_{i,\vx}^\top \vA(\vx_2)\vw_{i,\vx} \sigma''\left(f_i(\vz_2)\right)\Bigg)\Bigg].\notag
\end{align}
By \cref{assump2} we have inequality 
\begin{align} \label{eq:ineq}
    & a_1 \vw_{1,\vx}^\top \vA(\vx_1)\vw_{1,\vx} \sigma''\left(f_1(\vz_1)\right) - a_1 \vw_{1,\vx}^\top \vA(\vx_2)\vw_{1,\vx} \sigma''\left(f_1(\vz_2)\right) \notag\\
    =~& a_1 \varepsilon^2\frac{\vw_{1,\vx}^\top}{\varepsilon} \vA(\vx_1)\frac{\vw_{1,\vx} }{\varepsilon}\sigma''\left((\vw_i\cdot {\vz_1}+{b_i})/{\varepsilon}\right)\notag
    \\ &- a_1 \varepsilon^2\frac{\vw_{1,\vx}^\top }{\varepsilon}\vA(\vx_2)\frac{\vw_{1,\vx}}{\varepsilon} \sigma''\left((\vw_i\cdot {\vz_1}+{b_i})/{\varepsilon}\right)\notag \\
    \le~&M|a_1|(\|\vw_1\|_1+|b_1|)^2\left(\frac{1}{\varepsilon}\left|\vw_{1}\cdot(\vz_1-\vz_2)\right|\right).\notag
\end{align}
Substituting this back into the previous equation, we obtain 
\begin{align}
   & \mathbb{E}_{\Xi_n}\left[\sup_{\vz}\frac{1}{n}\sum \xi_i a_i \vw_{i,\vx}^\top \vA(\vx)\vw_{i,\vx} \sigma''\left(f_i\right)\right] \notag \\
   \leq ~& \frac{1}{2n}\mathbb{E}_{\xi_2,\ldots,\xi_m}\Bigg[\sup_{\vz_1,\vz_2}\Bigg(\frac{M}{\varepsilon}|a_1|(\|\vw_1\|_1+|b_1|)^2\left|\vw_{1}\cdot(\vz_1-\vz_2)\right| \notag \\
   &+ \sum_{i=2}^m\xi_i a_i \vw_{i,\vx}^\top \vA(\vx_1)\vw_{i,\vx} \sigma''\left(f_i(\vz_1)\right) + \sum_{i=2}^m\xi_i a_i \vw_{i,\vx}^\top \vA(\vx_2)\vw_{i,\vx} \sigma''\left(f_i(\vz_2)\right)\Bigg)\Bigg] \notag \\
   = ~& \frac{1}{2n}\mathbb{E}_{\xi_2,\ldots,\xi_m}\Bigg[\sup_{\vz_1,\vz_2}\Bigg(\frac{M}{\varepsilon}|a_1|(\|\vw_1\|_1+|b_1|)^2\left(\vw_{1}\cdot\vz_1-\vw_{1}\cdot\vz_2\right) \notag \\
   &+ \sum_{i=2}^m\xi_i a_i \vw_{i,\vx}^\top \vA(\vx_1)\vw_{i,\vx} \sigma''\left(f_i(\vz_1)\right) + \sum_{i=2}^m\xi_i a_i \vw_{i,\vx}^\top \vA(\vx_2)\vw_{i,\vx} \sigma''\left(f_i(\vz_2)\right)\Bigg)\Bigg] \notag\\
   = ~& \mathbb{E}_{\Xi_n}\Bigg[\sup_{\vz}\frac{1}{n}\Bigg(\varepsilon^{-1}|a_1|M(\|\vw_1\|_1+|b_1|)^2\xi_1\vw_{1}\cdot\vz \notag \\
   &+ \sum_{i=2}^m \xi_i a_i \vw_{i,\vx}^\top \vA(\vx)\vw_{i,\vx} \sigma''\left(f_i(\vz)\right)\Bigg)\Bigg], \notag
\end{align}
Repeating this process \(n\) times, we have 
\begin{align}
      &\mathbb{E}_{\Xi_n}\left[\sup_{\vz}\frac{1}{n}\sum \xi_i a_i \vw_{i,\vx}^\top \vA(\vx)\vw_{i,\vx} \sigma''\left(f_i\right)\right] \\ \le~&\mathbb{E}_{\Xi_n}\left[\sup_{\vz}\frac{1}{n}\sum_{i=1}^m\varepsilon^{-1}|a_i|M(\|\vw_i\|_1+|b_i|)^2\xi_i\vw_{i}\cdot\vz\right]\\ \leq~&\varepsilon^{-1}M\mathbb{E}_{\Xi_n}\left\|\frac{1}{n}\sum_{i=1}^m\xi_i|a_i|(\|\vw_i\|_1+|b_i|)^2\vw_{i}\right\|_1. \label{eq:term1}
\end{align} Similarly, we have \begin{align}
    \expt_{\Xi_n}\left[\sup_{\vz}\frac{1}{n}\sum a_i \xi_i \varepsilon (\nabla\cdot\vA(\vx))^\top \vw_i \sigma'\left(f_i\right)\right]\le \frac{M}{\varepsilon}\mathbb{E}_{\Xi_n}\left\|\frac{1}{n}\sum_{i=1}^m\xi_i|a_i|(\|\vw_i\|_1+|b_i|)^2\vw_{i}\right\|_1\label{eq:term2}
\end{align} 
It follows from \cref{eq:less1} that
the expectations of the following term in the summations of
\cref{eq:term1} and \cref{eq:term2} are less than $(1+\delta)\|\fN_\varepsilon\|_{\fB_{\varepsilon,\alpha^*}}$, i.e.,
\begin{align}
 \sqrt{\expt_{(a_i,\vw_i,b_i)\sim \rho}
  \norm{\vq_i}_1^2} \le (1+\delta)\|\fN_\varepsilon\|_{\fB_{\varepsilon,\alpha^*}}, \quad
 \vq_i \equiv
 |a_i|(\|\vw_i\|_1+|b_i|)^2\vw_{i}.
\end{align}Putting together, we have 
\begin{align}
&\expt_{(a_i,\vw_i,b_i)_{i=1}^n\sim \rho^n}\left [\sup_{\vz }\frac{1}{n}\sum\left({\varepsilon^{-2-\alpha^*}} \zeta_i - \fN_\varepsilon\right)\right] \\
\le~ & \expt_{(a_i,\vw_i,b_i)_{i=1}^n\sim \rho^n}\frac{4 M}{\varepsilon^{3-\alpha^*}}\expt_{\Xi_n}\left\|\frac{1}{n}\sum\xi_i\vq_i\right\|_1\\
\le~ & \frac{4 M\sqrt{2d}}{\varepsilon^{3-\alpha^*}}\expt_{(a_i,\vw_i,b_i)_{i=1}^n\sim \rho^n}\expt_{\Xi_n}\left\|\frac{1}{n}\sum\xi_i\vq_i\right\|_2 \\
\le~ &\frac{4 M\sqrt{2d}}{n\varepsilon^{3-\alpha^*}}\sqrt{\expt_{(a_i,\vw_i,b_i)_{i=1}^n\sim \rho^n}\expt_{\Xi_n}\left\|\sum\xi_i\vq_i\right\|_2^2},\end{align}where  
$\|\boldsymbol{z}\|_2\le\|\boldsymbol{z}\|_1 \leq \sqrt{2d}\|\boldsymbol{z}\|_2, \forall \boldsymbol{z} \in \mathbb{R}^{2d}$ is used,
and the last step is due to Jensen’s inequality. By the direct calculation, we know the cross-terms will disappear under the expectation, therefore, we have \begin{align}
&\expt_{(a_i,\vw_i,b_i)_{i=1}^n\sim \rho^n}\left [\sup_{\vz }\frac{1}{n}\sum\left({\varepsilon^{-2-\alpha^*}} \zeta_i - \fN_\varepsilon\right)\right] \\
\le~ &\frac{4 M\sqrt{2d}}{n\varepsilon^{3-\alpha^*}}\sqrt{\expt_{(a_i,\vw_i,b_i)_{i=1}^n\sim \rho^n}\sum\left\|\vq_i\right\|_2^2}\\
\le~ &\frac{4 M\sqrt{2d}}{n\varepsilon^{3-\alpha^*}}\sqrt{\expt_{(a_i,\vw_i,b_i)_{i=1}^n\sim \rho^n}\sum\left\|\vq_i\right\|_1^2}\\
\le~ &\frac{4(1+\delta) M\sqrt{2d}\|\fN_\varepsilon\|_{\fB_{\varepsilon,\alpha^*}}}{\varepsilon^{3-\alpha^*}\sqrt{n}}.\label{eq:expplus}\end{align}
Similarly, we can show 
\begin{align}
    &\expt_{(a_i,\vw_i,b_i)_{i=1}^n\sim \rho^n}\left [-\sup_{\vz}\frac{1}{n}\sum\left({\varepsilon^{-2+\alpha^*}} \zeta_i - \fN_\varepsilon\right)\right]
    \le\frac{4(1+\delta) M\sqrt{2d}\|\fN_\varepsilon\|_{\fB_{\varepsilon,\alpha^*}}}{\varepsilon^{3-\alpha^*}\sqrt{n}},
\end{align}
from which and \cref{eq:expplus}, we finally can show that
\begin{align}    \label{first}
\expt_{(a_i,\vw_i,b_i)_{i=1}^n\sim\rho^n}
\sup_{\vz}\left|\frac{1}{n}\sum\left({\varepsilon^{-2+\alpha^*}} \zeta_i - 
\fN_\varepsilon 
\right)\right|
&\le\frac{4(1+\delta)M\sqrt{2d}\|\fN_\varepsilon\|_{\fB_{\varepsilon,\alpha^*}}}{\varepsilon^{3-\alpha^*}\sqrt{n}},
\end{align}
and furthermore
\begin{align}
&\expt_{(a_i,\vw_i,b_i)_{i=1}^n\sim \rho^n}
\frac{1}{n}\sum
\left(|a_i|\left(\|\vw_i\|_1^3+2\|\vw_i\|_1^2|b_i|+\|\vw_i\|_1|b_i|^2\right)\right)^2
\le (1+\delta)^2\|\fN_\varepsilon\|^2_{\fB_{\varepsilon,\alpha^*}}.
\end{align}
Denote by $\fE_1$ and $\fE_2$ the two terms in the expectations, i.e.,
\begin{align}
\fE_1(\left\{a_i,\vw_i,b_i\right\}) &= \sup_{\vz}\left|\frac{1}{n}\sum\left({\varepsilon^{-2+\alpha^*}} \zeta_i - 
\fN_\varepsilon 
\right)\right|
\\
\fE_2(\left\{a_i,\vw_i,b_i\right\}) &= \frac{1}{n}\sum
\left(|a_i|\left(\|\vw_i\|_1^3+2\|\vw_i\|_1^2|b_i|+\|\vw_i\|_1|b_i|^2\right)\right)^2,
\end{align}
and define two events 
\begin{align}
E_1 = \left\{ \fE_1 < \frac{8 M\sqrt{2d}\|\fN_\varepsilon\|_{\fB_{\varepsilon,\alpha^*}}} 
{\varepsilon^{3-\alpha^*}\sqrt{n}}\right\}, \quad
E_2 = \left\{ \fE_2 < 3\|\fN_\varepsilon\|^2_{\fB_{\varepsilon,\alpha^*}}\right\}.
\end{align}
Based on Markov's inequality, we have
\begin{align}
    \rmP(E_1) &\ge 1 - \frac{\expt (\fE_1)}{\frac{8 M\sqrt{2d}\|\fN_\varepsilon\|_{\fB_{\varepsilon,\alpha^*}}}{\varepsilon^{3-\alpha^*}\sqrt{n}}} > \frac{7}{16}, \notag \\
    \rmP(E_2) &\ge 1 - \frac{\expt (\fE_2)}{3\|\fN_\varepsilon\|^2_{\fB_{\varepsilon,\alpha^*}}} > \frac{37}{64},
\end{align}
so that \(\rmP(E_1 \cap E_2) > \frac{1}{64}\), indicating that there exist 
\((a_i, \vw_i, b_i)_{i=1}^n\) such that \cref{first_part} and \cref{order1}
are true. We can collect all such weights and construct the NNs
in $\mathcal{A}_n$.
\end{proof}
Based on \cref{bound}, we know that when \(\alpha^*\) satisfies \(\|\mathcal{N}_\varepsilon(\vx,\vy)\|_{\mathcal{B}_{\varepsilon,\alpha^*}} \leq 1\), it ensures that the average magnitude of the parameters is not large. However, we also do not want to choose a too small $\alpha^{*}$, as that might result in a larger $n$ reaching the same approximation accuracy due to the increase in the denominator in the upper bound of \cref{first_part}. Therefore, we suggest choosing an \(\alpha^*\) such that \(\|\mathcal{N}_\varepsilon(\vx,\vy)\|_{\mathcal{B}_{\varepsilon,\alpha^*}} \approx 1\). Based on \cref{twobarron}, we know that for a fixed \(\varepsilon\), a larger \(\alpha^*\) results in a larger Barron norm. Therefore,  we can pick the following \(\alpha\):
\begin{equation}
    \alpha = \sup \left\{\alpha^* \mid \|\mathcal{N}_\varepsilon(\vx,\vy)\|_{\mathcal{B}_{\varepsilon,\alpha^*}} \leq 1\right\}.\label{definealpha}
\end{equation}
Based on \cref{twobarron}, for any fixed \(\varepsilon > 0\), there are either infinite elements or none in 
\[
\left\{\alpha^* \mid \|\mathcal{N}_\varepsilon(\vx,\vy)\|_{\mathcal{B}_{\varepsilon,\alpha^*}} \le 1\right\}.
\]
If \(\left\{\alpha^* \mid \|\mathcal{N}_\varepsilon(\vx,\vy)\|_{\mathcal{B}_{\varepsilon,\alpha^*}} \le 1\right\} = \emptyset\), we define \(\alpha = -\infty\). This indicates that no large-scale NNs can even roughly approximate the Green's function of \cref{poisson} for \(\beta = 1\). This issue arises from the excessive complexity of \(\vA(\vx)\) and \(\Omega\), scenarios which will not be addressed in this paper. 
If there are infinite elements in that set, there must be an upper bound based on  
\cref{twobarron}; therefore, \(\alpha\) cannot be \(+\infty\). By choosing such \(\alpha\), we can ensure that  $n$ in the large-scale network can be small. 

\cref{bound} provides the approximation property of the large-scale NN. Next, we analyze the approximation property of the MSNN. Before proceeding to \cref{bound2}, we discuss an assumption in \cref{bound2} in more detail. In the theorem, we assume
\[
r_n(\vx,\vy) = \nabla \cdot \left( \vA(\vx) \nabla \phi_1(\vz; \vtheta_1) \right) - \mathcal{N}_\varepsilon(\vx,\vy)
\]
is a $1$-$\vA(\vx)$-Barron function with respect to $\sigma(x)$.  This assumption is reasonable due to \cite{wojtowytsch2022representation}, which shows that sufficiently smooth functions belong to Barron-type spaces. Here, for any \(\varepsilon > 0\), \(\mathcal{N}_\varepsilon\) is smooth and \(\nabla \cdot \left( \vA(\vx) \nabla \phi_1(\vz; \vtheta_1) \right)\) is also smooth if \(\vA\) and \(\sigma\) are smooth. Without being rigorous, we treat this as an assumption.

\begin{thm}\label{bound2}
Under the same assumptions in \cref{bound}, if there is a large-scale 
NN, denoted by $\phi_1(\vz; {\vtheta_1}) \in\mathcal{A}_n$, such that 
$$r_n(\vx,\vy)=\nabla \cdot \left( \vA(\vx) \nabla \phi_1(\vz; {\vtheta_1}) \right) - \mathcal{N}_\varepsilon(\vx,\vy)$$
is a $1$-$\vA(\vx)$-Barron function with respect to $\sigma(x)$,  for any $m \ge 1$, there exists a small-scale NN with activation function $\sigma(x)$, 
\[
\phi_2(\vz; \vtheta_2) = \sum_{i=n+1}^{m+n}  \frac{1}{m}a_i \sigma\left(\vw_i \cdot \vz +  b_i\right)
\]
such that for $\phi(\vz, \vtheta)=\phi_1(\vz, \vtheta_1) + \phi_2(\vz, \vtheta_2)$ in \cref{nm}, we have
\begin{equation}
\sup_{\vx,\vy} \left| \nabla \cdot \left( \vA(\vx) \nabla \phi(\vz; \vtheta) \right) - \mathcal{N}_\varepsilon(\vx,\vy)\right| \le 
\|r_n(\vx,\vy)\|_{\mathcal{B}_1} \frac{8 M \sqrt{2d}}{\sqrt{m}}
\label{second_part}
\end{equation}
and 
\begin{equation}
\frac{1}{m} \sum_{i=n+1}^{n+m} \left(|a_i| \left(\|\vw_i\|_1^3 + 2\|\vw_i\|_1^2 |b_i| + \|\vw_i\|_1 |b_i|^2\right)\right)^2 \le 3\|r_n(\vx,\vy)\|^2_{\mathcal{B}_1}.\label{order2}
\end{equation}
\end{thm}
\begin{proof}
The proof is similar to that of \cref{bound}, and is therefore omitted here.
\end{proof}

Theorems \ref{bound} and \ref{bound2} provide the approximation properties of the MSNN and illustrate the advantages of employing the MSNN in \cref{nm} over single-scale NNs. If only a large-scale NN \(\phi_1(\vz; \vtheta_1)\) is used to approximate the Green's function,  \(\mathcal{O}\left( (\varepsilon^{-\alpha+3} \widehat{\epsilon})^{-2} \right)\) parameters are required to achieve an error of \(\mathcal{O}(\widehat{\epsilon})\) according to \cref{first_part}. If we only have a small-scale NN, i.e., \(\phi_1 = 0\) in \cref{bound2}, the magnitude of the parameters will be large due to \cref{order2} and the fact that \(\|\mathcal{N}_\varepsilon\|_{\fB_1}\) is large. 
On the other hand, by adopting the MSNN approach, it is possible to use only \(\mathcal{O}\left( \varepsilon^{2(\alpha - 3)} + \widehat{\epsilon}^{-2} \right)\) parameters to achieve an error of \(\mathcal{O}(\widehat{\epsilon})\).
The reason is that when using MSNN with $m \ge 1$ and under the assumption that 
$\sup_{\vx,\vy} \left|r_n(\vx, \vy)\right|=\fO(1)$,
we can possibly have the estimate $\norm{r_n(\vx, \vy)}_{\fB_1}=\fO(1).$
Then the upper bound in \cref{second_part} can be simplified as
\begin{equation}\label{2order}
\sup_{\vz} \left| \nabla \cdot \left( \vA(\vx) \nabla \phi(\vz; \vtheta) \right) - 
\mathcal{N}_\varepsilon(\vx,\vy)
\right| 
\le \frac{C_2}{\sqrt{m}},
\end{equation}
 where \(C_2\) is $\fO(1)$ respect to \(\varepsilon\).

For \(\|r_n(\vx, \vy)\|_{\mathcal{B}_1} = \mathcal{O}(1)\) to hold, two conditions must be met. First, \(r_n(\vx, \vy)\) must lie in Barron-type spaces, which requires the function to have a smooth structure. This was discussed before \cref{bound2} and shown to be reasonable for \(r_n(\vx, \vy)\). Second, the norm should be of order one. Based on \cite{weinan2022representation,chen2021representation}, we know that the \(H^s\)-norm, when \(s > \frac{d}{2} + 3\), can control the Barron-type norm, and the $H^s$-norm can be controlled by the $C^s$-norm \cite{evans2022partial}, where 
\[
\|f\|_{C^s\left([0,1]^d\right)}:=\sum_{\|\boldsymbol{\alpha}\|_1 \leq s}\left\|\partial^{\valpha} f\right\|_{C^{0}\left([0,1]^d\right)}.
\]
When the \(C^s\)-norm has the same scale as the \(C^0\)-norm, we can infer that when the \(C^0\)-norm is of order one, the Barron-type norm can be controlled by an order one term. This assumption may not hold for a general target function; however, it is reasonable in our case. 
For the region where $\vx$ is close to $\vy$, \(\nabla \cdot \left( \vA(\vx) \nabla \phi_1(\vz; \vtheta_1) \right)\) will reduce the high-frequency components of \(\mathcal{N}_\varepsilon\), ensuring that in this region, the \(C^0\)-norm is comparable to the \(C^s\)-norm. This is the motivation for establishing the large-scale neural network first. For the region where $\vx$ is away from $\vy$, since \(\mathcal{N}_\varepsilon\) is smooth and  \(\nabla \cdot \left( \vA(\vx) \nabla \phi_1(\vz; \vtheta_1) \right)\) does not explode when the activation function is local, such as \(\exp(-x^2)\) or \(\tanh(x)\), the \(C^0\)-norm has also the same scale as the \(C^s\)-norm. As a result, for \(r_n\), the \(C^0\)-norm can have the same scale as the Barron-type norm. A rigorous proof of this assumption would require a constructive approach, rather than the existence proof presented in this paper, especially regarding how to establish \(\phi_1(\vz; \vtheta_1)\) to achieve the goal stated above, and not merely meet the error bound in \cref{bound}. We will explore this in future work.

\cref{bound2} also shows that although various large-scale NNs can effectively reduce the approximation error of Green's functions, merely reducing the error scale does not guarantee that the small-scale NN will successfully learn the residual, particularly if the residual is not smooth and displays noise-like characteristics.  \cref{bound2} asserts that successful learning of the residual is more likely if it resides within the Barron-type space, suggesting a higher degree of smoothness, as outlined in \cite{barron1993universal,lu2021priori}. In practical implementations, to identify suitable large-scale NNs, we can initially train a large-scale NN to approximate the Green's function, then add a small-scale NN  and finally conduct simultaneous training of both scales. This dual-scale training approach not only further refines the large-scale NN to pinpoint the correct large-scale components but also improves the small-scale NN's capacity to capture finer details, thereby ensuring thorough learning across various scales.

{Combining (ii) in Propositions \ref{prop:u_eps_convergence} and \ref{prop:NN_vs_Geps}, and Theorem \ref{bound2} yields a complete error estimate for the multiscale–network approximation of the Green's function.

\begin{cor}\label{cor:multiscale_error}
Assume the hypotheses of Theorem \ref{bound2}.  
Let $\varepsilon>0$, and suppose $\vA=(a_{ij})_{i,j=1}^{d}\in W^{1,\infty}(\Omega;\mathbb R^{d\times d})$ is uniformly elliptic, where $\Omega\subset[0,1]^{d}$ is a bounded open domain.  
Let $\phi(\vx,\vy;\boldsymbol\theta)$ be the multiscale neural network constructed in Theorem \ref{bound2}, satisfying the homogeneous Dirichlet condition in $\vx$ for every fixed $\vy\in\Omega$ defined in Theorem \ref{bound2}.  
For any $f\in H^{1}_{0}(\Omega)\cap L^{\infty}(\Omega)$ define
\[
    u(\vx)=\int_{\Omega}G(\vx,\vy)\,f(\vy)\,\D\vy,
    \qquad
    u_{\boldsymbol\theta}(\vx)=\int_{\Omega}\phi(\vx,\vy;\boldsymbol\theta)\,f(\vy)\,\D\vy,
\]
where $G$ is the Green's function of \eqref{poisson}.  
Then $u_{\boldsymbol\theta}\in H^{1}_{0}(\Omega)$ and
\[
    \|u-u_{\boldsymbol\theta}\|_{H^{1}(\Omega)}
    \;\le\;
    C\,\varepsilon\,\|f\|_{H^{1}(\Omega)},
\]
with a constant $C>0$ independent of $f$ and $\varepsilon$. The network
contains
\(
   N=\mathcal O\!\bigl(\varepsilon^{\,2(\alpha-3)}+\varepsilon^{-2}\bigr)
   ~\text{parameters},
\)
where
$\alpha$ is defined in \eqref{definealpha}.
\end{cor}}

\subsection{Training strategies}
\cref{bound} and \cref{bound2} suggest to first train the large-scale network $\phi_1$ to approximately solve  \cref{approximatedd}, i.e., 
\begin{equation}
   \nabla\cdot(\vA(\vx)\nabla \phi_1)\approx \mathcal{N}_\varepsilon(\vx,\vy)=\left(\frac{1}{\varepsilon\sqrt{\pi}}\right)^d\exp\left(-\frac{\Vert\vx-\vy\Vert_2^2}{\varepsilon^2}\right).
   \end{equation}
Given the form of $\phi_1$ defined in \cref{eq:ph11} and \cref{eq:diffphi0}, we further have
\begin{align}
\nabla\cdot(\vA(\vx)\nabla \phi_1) &= 
\frac{\varepsilon^{\alpha-2}}{n}
\sum_{i=1}^{n}\vw_{i,\vx}^\top \vA(\vx)\vw_{i,\vx}a_i\sigma''\left(\vw_i \cdot (\varepsilon^{-1}\vx,\varepsilon^{-1}\vy) + \varepsilon^{-1} b_i\right) \\
&+\frac{\varepsilon^{\alpha-2}}{n}\sum_{i=1}^{n}\varepsilon(\nabla\cdot\vA(\vx))^\top \vw_{i,\vx} \sigma'\left(\vw_i \cdot (\varepsilon^{-1}\vx,\varepsilon^{-1}\vy) + \varepsilon^{-1} b_i\right).
\end{align}

In the previous section, we propose one way to choose an  $\alpha$ in \cref{definealpha} under the condition that $\beta = 1$ for the large-scale NN. 
However, this choice does not yield an explicitly computable formula for $\alpha$. 
A practical heuristic for selecting $\alpha$ is to ensure that the magnitude of the NN parameters has roughly the same order and is bounded, which makes the NN easier to train, particularly when $\varepsilon$ is small, as shown in \cite{liu2020multi,wang2020multi,zhou2013causal}. If all the parameters are of $\mathcal{O}(1)$, 
apparently it follows that $\nabla\cdot(\vA(\vx)\nabla \phi_1)$ is of
$\mathcal{O}(\varepsilon^{\alpha-2})$. 
Therefore, to match the order of the right-hand side,
$\mathcal{N}_{\varepsilon}(\vx,\vy)$, which is $\mathcal{O}(\varepsilon^{-d})$,
we should set $\alpha=2-d$.



The structure of the MSNN we use in practice is shown in \cref{fig:net}, following a similar approach by Teng et al. \cite{teng2022learning}.
We include an additional input feature 
that is the distance of  coordinates $\vx$ and $\vy$, i.e.,
$\vx - \vy$, as the value of the Green's function generally depends on the distance between $\vx$ and $\vy$ \cite{teng2022learning}.
The training process begins with the training of the large-scale network, followed by a simultaneous training of both networks of the large and small scales.

\begin{figure}[htbp]
    \centering
    \includegraphics[width=0.7\linewidth]{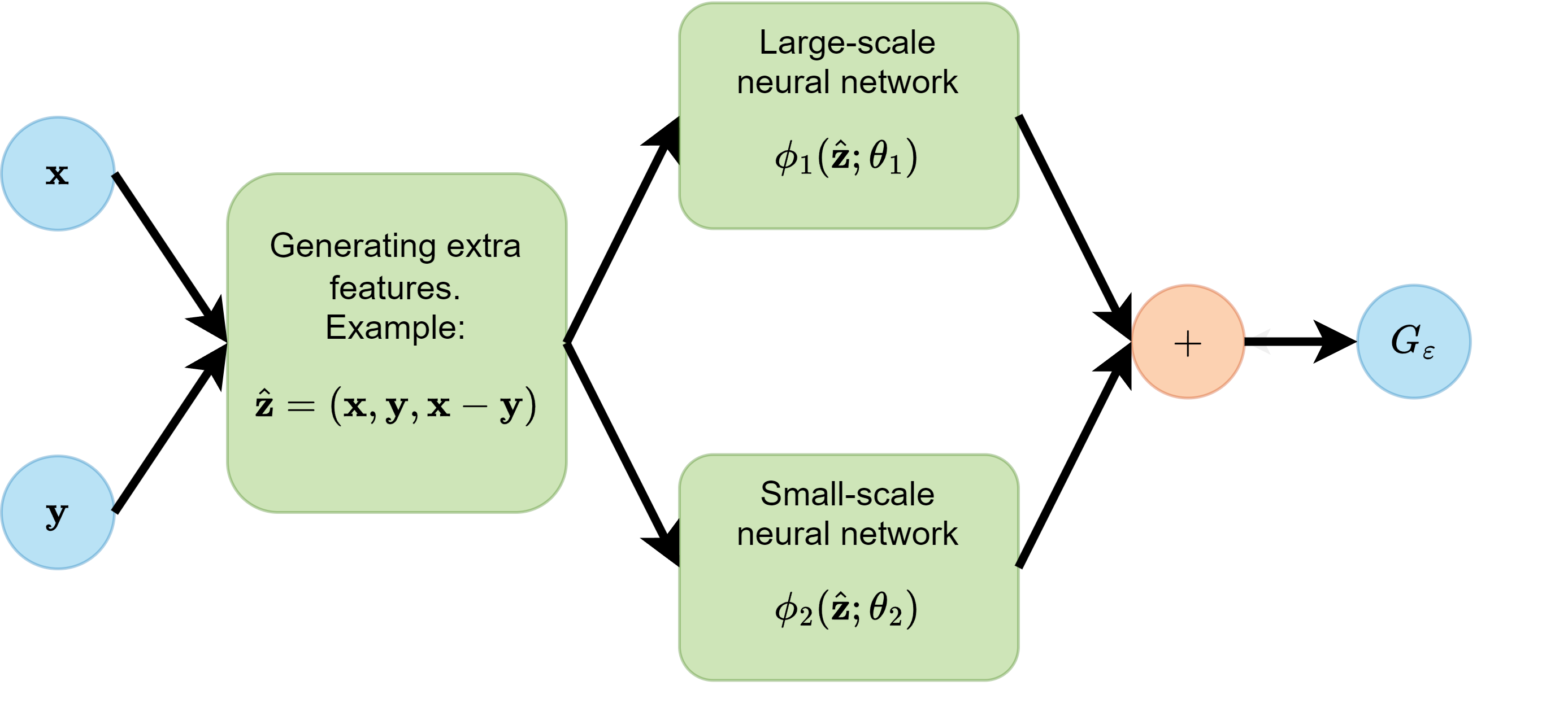}
    \caption{The design of the multiscale NN, incorporating $\vx - \vy$ as an additional feature.}
    \label{fig:net}
\end{figure}

As training a NN across the entire domain $\Omega$ can be challenging, we employ the domain decomposition (DD) strategy proposed by Teng et al. \cite{teng2022learning}. 
We divide $\Omega$ into $p$ disjoint subdomains such that $\Omega=\bigcup\Omega_i$ for $i=1,2,\ldots,p$, and train one NN $\phi^i((\vx,\vy),\vtheta^i)$ for each subdomain $\Omega\times\Omega_i$.
Instead of using a naive partitioning as in \cite{teng2022learning}, we exploit graph partitioners as often used in general sparse matrix packages \cite{xu2024parallel,tianshi} in order to handle general domains, by which 
we divide 1D domains into line segments and 2D domains into triangular elements. 
We denoted by $\mathcal{T}_C$ the mesh of $\Omega$,
which does not need to be fine, as we do not need the subdomains to be perfectly balanced. The mesh $\mathcal{T}_C$ can be associated with a graph $\mathcal{G}_C(\mathcal{V},\mathcal{E})$, where the vertex set $\mathcal{V}$ represents the elements in $\mathcal{T}_C$, 
and two vertices are connected by an edge in $\mathcal{E}$ if the corresponding elements are neighbors. 
Then, a $p$-way  partition with edge separators is applied to $\mathcal{G}_C$, dividing it into $p$ subgraphs.
Finally, a subdomain is formed by the union of the elements corresponding to the subgraph. 
An illustration of the DD for a rectangular domain is provided in the left panel in \cref{fig:dd}.

\begin{figure}[htbp]
    \centering
    \includegraphics[width=0.3\linewidth]{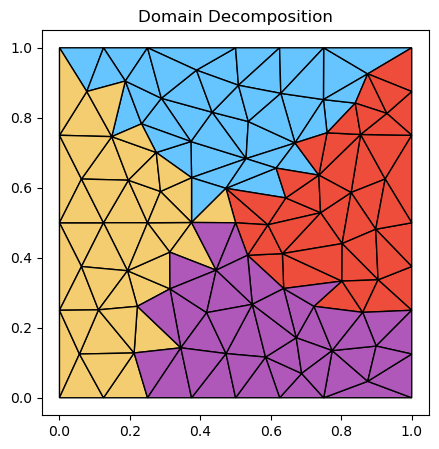}
    \includegraphics[width=0.3\linewidth]{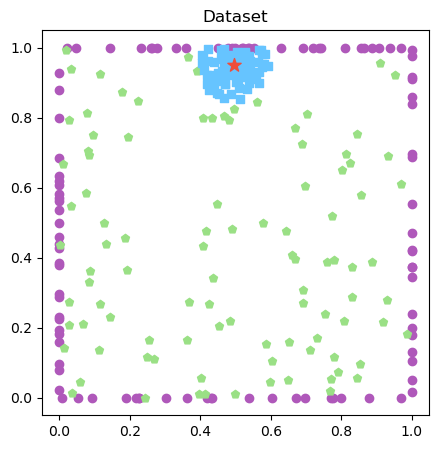}
    \includegraphics[width=0.3\linewidth]{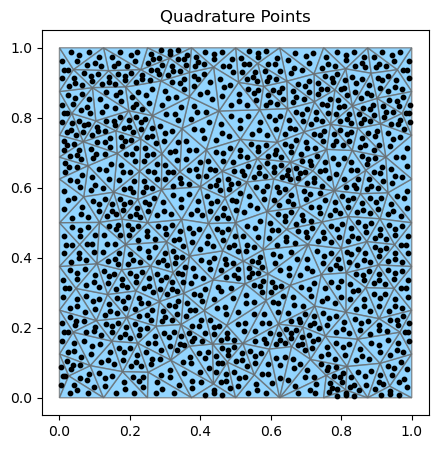}
    \caption{DD, dataset, and quadrature points in a rectangular domain. Left: DD using a coarse mesh $\mathcal{T}_C$ and $4$-way graph partitioning. Middle: dataset with $\vy$ from one subdomain (the red star), $\vx$ close to $\vy$ (the blue squares), $\vx$ uniformly sampled in $\Omega$ (the green pentagon), and $\vx$ on boundary (the purple dots). Right: Quadrature points (the black dots) used in a fine mesh $\mathcal{T}_F$.}
    \label{fig:dd}
\end{figure}

The training data are generated as follows. 
For subdomain $i$, we sample a few points $S^i\in\Omega_i$, and for each $\vy\in  S^{i}$, we sample $\vx$ points to form the following sets of $(\vx, \vy)$ pair: 1) $S_{\vy,bdry}^i$ with $\vx$ on the domain boundary $\partial\Omega$; 2)  $S_{\vy,far}^i$ with $\vx$ uniformly distributed within the domain $\Omega$; and 3)  $S_{\vy,near}^i$ with $\vx$ uniformly distributed within a predefined radius from $\vy$ in the domain $\Omega$. 
The dataset associated with $i$-th subdomain is constructed as
\begin{align}
    & S_{bdry}^{i} := \cup_{\vy\in S^i}S_{\vy,bdry}^i\\
    & S_{int}^{i} := \left(\cup_{\vy\in S^i}S_{\vy,near}^i\right)\cup\left(\cup_{\vy\in S^i}S_{\vy,far}^i\right),
\end{align}
where $S_{bdry}^{i}$ are data points on the boundary and $S_{int}^{i}$ are data points within the domain. 
The data points sampled associated with one $\vy$ are illustrated in the middle panel in \cref{fig:dd}.

Defined on these data points, the loss function $L^i$ on the $i$-th subdomain is defined as
\begin{equation}
    L^i = w_{bdry}L_{bdry}^i + w_{res}L_{res}^i + w_{sym}L_{sym}^i,
\end{equation}
which was first proposed in \cite{teng2022learning}, where 
\begin{align}
    & L^i_{bdry}(\vtheta^i) = \frac{1}{|S_{bdry}^{i}|}\sum_{(\vx,\vy) \in S_{bdry}^{i}} \phi^i((\vx,\vy),\vtheta^i)^2 \\
    & L^i_{res}(\vtheta^i) = \frac{1}{|S_{int}^{i}|}\sum_{(\vx,\vy)\in S_{int}^{i}} \left[\mathcal{L}\phi^i((\vx,\vy),\vtheta^i)-\mathcal{N}_{\varepsilon}(\vx,\vy)\right]^2 \\
    & L^i_{sym}(\vtheta^i) = \frac{1}{|S_{int}^{i}|}\sum_{(\vx,\vy) \in S_{int}^{i}} \left[\phi^i((\vx,\vy),\vtheta^i)-\phi^i((\vy,\vx),\vtheta^i)\right]^2.
\end{align}
Here, $L^i_{bdry}$ represents the error on the boundary, $L^i_{res}$ measures the point-wise residual of the PDE, and $L^i_{sym}$ is designed to ensure the symmetry $\phi^i(x,y)=\phi^i(y,x)$.
The weight terms $w_{bdry}$, $w_{res}$, and $w_{sym}$ are used to balance these loss terms.
The NNs on those subdomains are trained independently in parallel.

After the model is trained, numerical integration is computed for the final solution.
The quadrature points are selected from a much finer mesh of the domain $\Omega$, which is denoted by $\mathcal{T}_F$.
The numerical integration uses the Gauss-Legendre rule for line segments and Dunavant’s integration rules for triangular elements \cite{dunavant1985high}.

The overall procedure for solving PDE problems using the multiscale NN learning approach is illustrated in \cref{fig:diagram}. 
Before concluding this section, it is worth noting that the analysis can be easily generalized to more complex linear differential operators. In \cref{num}, we present numerical examples with more intricate operators to demonstrate the effectiveness of our method.

\begin{figure}[htb]
    \centering
    \includegraphics[width=0.95\linewidth]{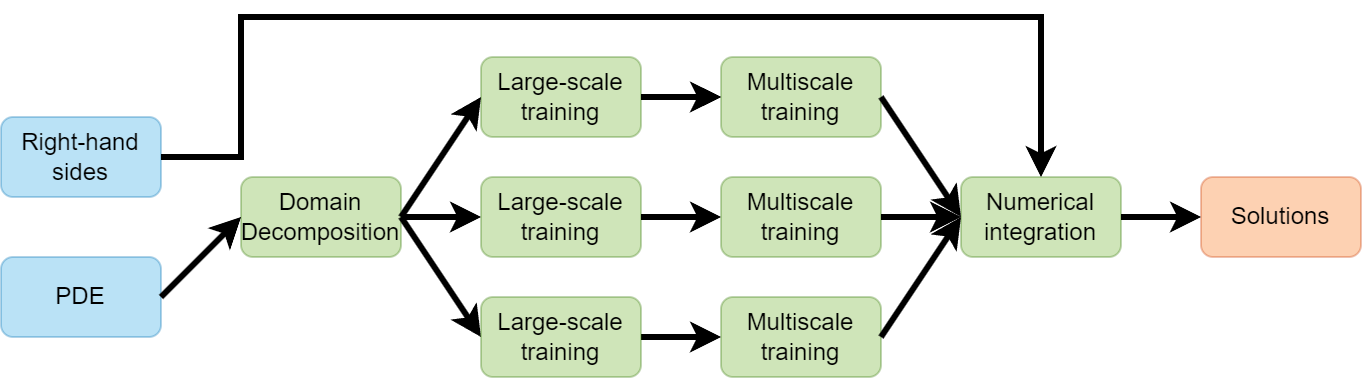}
    \caption{The overall procedure of solving PDEs by the multiscale NN learning approach
for Green's function. A large-scale network
    in each subdomain is first
    trained, and both the large- and small-scale NNs in the subdomain are trained together.
    The PDE solution is computed by numerical integration with the learned Green's function and the right-hand-side function
    of the PDE.
    \label{fig:diagram}}
\end{figure}

\section{Numerical Experiments}\label{num}
This section outlines the experiments conducted using our MSNN approach to approximate Green’s functions. 
We specifically compare the performance of our multiscale model (in the form of $\phi_1+\phi_2$) to the single scale models (in the form of $\phi_1$ or $\phi_2$ alone) for various problems, including convergence in the training process and prediction accuracy.
The following notations are adopted in this section:
\begin{itemize}
\item $\phi^l$: the standalone large-scale network model.
\item $\phi^s$: the standalone small-scale network model.
\item $\phi^m_1$: the large-scale network within the multiscale framework.
\item $\phi^m_2$: the small-scale network within the multiscale framework.
\item $\phi^m$: the multiscale network (combined output of $\phi^m_1$ and $\phi^m_2$).
\end{itemize}


Our experiment code was developed using \texttt{Python} with \texttt{PyTorch} \cite{NEURIPS2019_bdbca288} for the NN implementations and \texttt{NumPy} \cite{harris2020array} for data generation. 
All the NNs used in our experiments are fully connected with linear layers and the {$\tanh$ activation function}. 
The experiments were conducted on a machine equipped with dual 24-core 2.1 GHz Intel Xeon Gold 5318Y CPUs, an NVIDIA H100 GPU, and 1 TB of RAM. 
The versions of the main software/libraries used were \texttt{Python} 3.8.19, \texttt{PyTorch} 2.2.2, \texttt{NumPy} 1.24.3, and \texttt{CUDA} 12.5.
Triangular elements were generated using the \texttt{MeshPy} package \cite{kloeckner2018meshpy}, and we implemented a simple spectral graph partitioning algorithm for graph partitioning.
When sampling data, the radius for sampling ``near points'' is set at $2\varepsilon$. 
In all the experiments, we fix the loss weights $w_{res}$ to $1.0$ and set $w_{bdry}=w_{sym}=\varepsilon^{-d}$.


\subsection{Model Problems}
We consider the following model problem:
\begin{equation}
\begin{cases}
- \Delta u(\vx) - c(\vx)u(\vx) = f(\vx), & \vx \in \Omega \\
u(\vx) = 0, & \vx \in \partial \Omega
\end{cases}\label{eq:model problem}
\end{equation}
We explore various domains $\Omega$ in $\mathbb{R}$ and $\mathbb{R}^2$, with different functions $c(\vx)$ from $\mathbb{R}^d$ to $\mathbb{R}$. 
We benchmark our network’s performance by comparing the numerical solutions obtained from our trained networks against those derived using the Finite Element Method (FEM) with linear elements. 
These comparisons are facilitated using the \texttt{FEniCS} package for FEM \cite{alnaes2015fenics,logg2012automated}.
We refine the mesh $\mathcal{T}_F$ using $\texttt{FEniCS}$ and obtain numerical solutions of high accuracy for comparison.

\subsection{Selection of $\varepsilon$}
Before testing our framework, we first study the influence of $\varepsilon$ on the approximation accuracy of the Green's function since we are using $\mathcal{N}_{\varepsilon}(\vx,\vy)$ as the approximation of {$\delta(\vx-\vy)$ distribution}.

In \cref{fig:gauss seg fixy} and \cref{fig:gauss seg all}, we compare the approximations using various $\varepsilon$ with the exact Green's function
for the 1-D problem with $c=0$ on $\Omega=[0,1]$.
The exact Green’s function is given by
\begin{equation}
G(x,y)=
\begin{cases}
x(1-y), & x < y \\
y(1-x), & x \geq y.
\end{cases}\label{eq:1D Greens}
\end{equation}
We first fix $y=0.95$ and compare the exact $G(x, 0.95)$ with the approximations obtained from FEM and various $\varepsilon$.
As shown in \cref{fig:gauss seg fixy}, the approximation with $\varepsilon=10^{-2}$ is very close to the true solution, whereas larger $\varepsilon$ values cannot yield accurate enough approximations.
In \cref{fig:gauss seg all}, we plot the entire Green's function.
Again, a good approximation can be obtained with $\varepsilon=10^{-2}$.

\begin{figure}[htbp]
    \centering
    \includegraphics[width=0.5\linewidth]{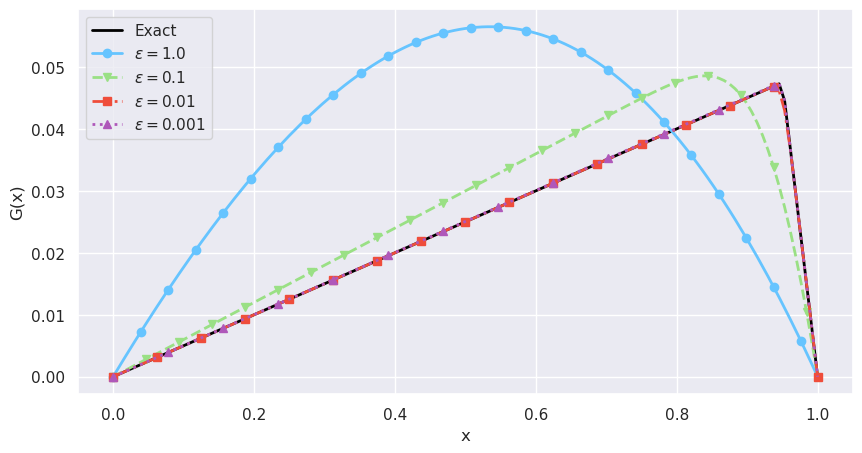}
    \caption{FEM Green's function approximations at $y=0.95$ using $\varepsilon = 1.0, 0.1, 0.01, 0.001$ for the 1-D problem \cref{eq:model problem} with $c=0$ on $\Omega=[0,1]$.}
    \label{fig:gauss seg fixy}
\end{figure}

\begin{figure}[htbp]
    \centering
    \includegraphics[width=0.95\linewidth]{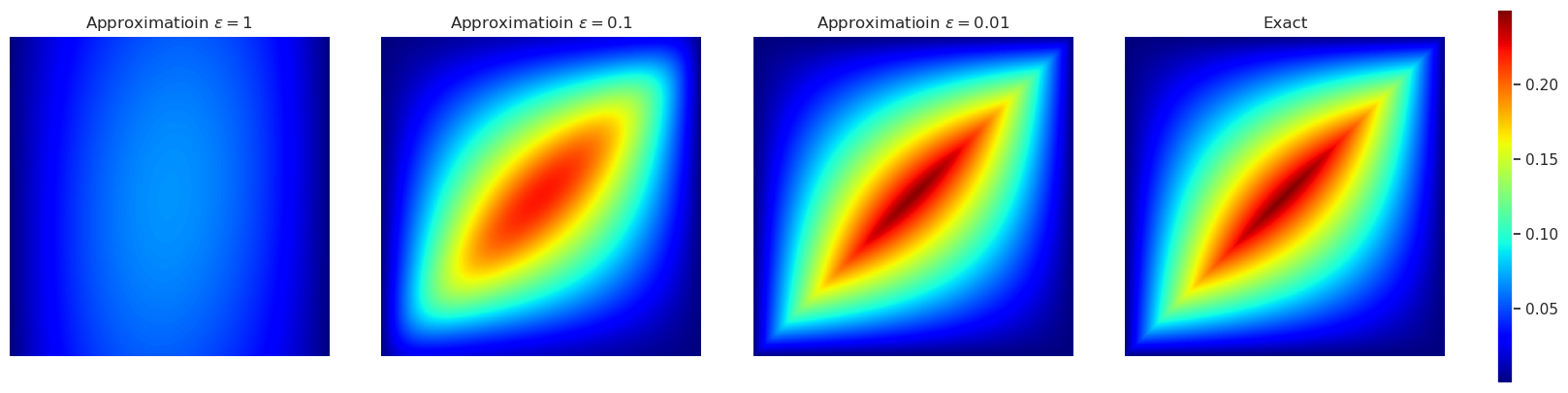}
    \caption{FEM Green's function approximations using $\varepsilon = 1.0, 0.1, 0.01$ for the 1-D problem \cref{eq:model problem} with $c=0$ on $\Omega=[0,1]$. }
    \label{fig:gauss seg all}
\end{figure}

\begin{figure}[htbp]
    \centering
    \includegraphics[width=0.95\linewidth]{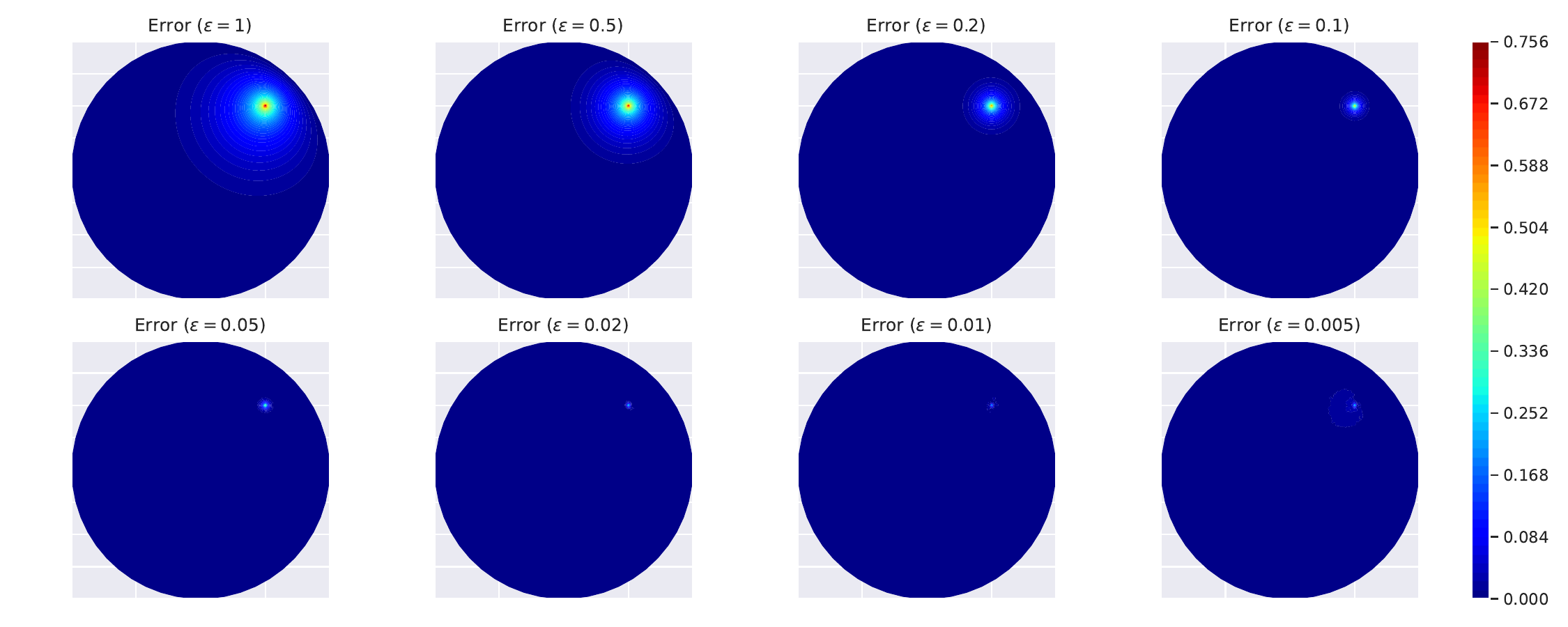}
    \caption{Difference between the FEM solution of Green's function approximation and the exact Green's function for the 2-D problem \cref{eq:model problem} in a unit circle with $c=0$ at $\vy=(0.5,0.5)$.}
    \label{fig:gauss circ}
\end{figure}

Next, we show the Green's function for 2-D problems on a unit circle with fixed $\vy=(0.5,0.5)$.
As we can see from  \cref{fig:gauss circ}, 
it is necessary to use a small $\varepsilon$ (e.g., lower than $0.02$) to obtain a good approximation. 
In our tests, we use $\varepsilon=0.01$ for 1D problems and $\varepsilon=0.02$ for 2D problems, unless otherwise noted.


\subsection{Magnitude of Network Parameters}

In our next experiment, we allow all networks to undergo enough Adam steps to reach a certain accuracy. We plot the histogram of network parameters for each network to justify that we are able to find a multiscale network such that the parameters are bounded with respect to $\varepsilon$.
We run a simple test with $c(x)=1+x^2$ on $\Omega=[0,1]$ with zero boundary condition and only learn the Green's function at a fixed $y=0.95$.
We choose this simple problem  because it is difficult to train the standalone small-scale network $\phi^s$ to converge to a reasonable accuracy for more challenging problems.
In this test, the NNs are configured with just one hidden layer to align with the theoretical expectations.
For multiscale network $\phi^m$, both the large-scale network $\phi^m_1$ and the small-scale network $\phi^m_2$ utilize a hidden layer of $100$ neurons. 
The standalone networks $\phi^l$ and $\phi^s$ are equipped with a hidden layer comprising $200$ neurons, ensuring that the total number of parameters in $\phi^m$, $\phi^l$ and $\phi^s$ remains close for a fair comparison.
We run Adam and check the error every $20,000$ iterations until the $l_{\infty}$ error on the testing set is smaller than $0.01$, which means that the network has a reasonable accuracy.
We test multiple $\varepsilon$ values from $0.01$ to $1.0$.
The results are shown in \cref{fig:seg_fixy_vlap_veps}.

\begin{figure}[htbp]
    \centering
    \includegraphics[width=0.95\linewidth]{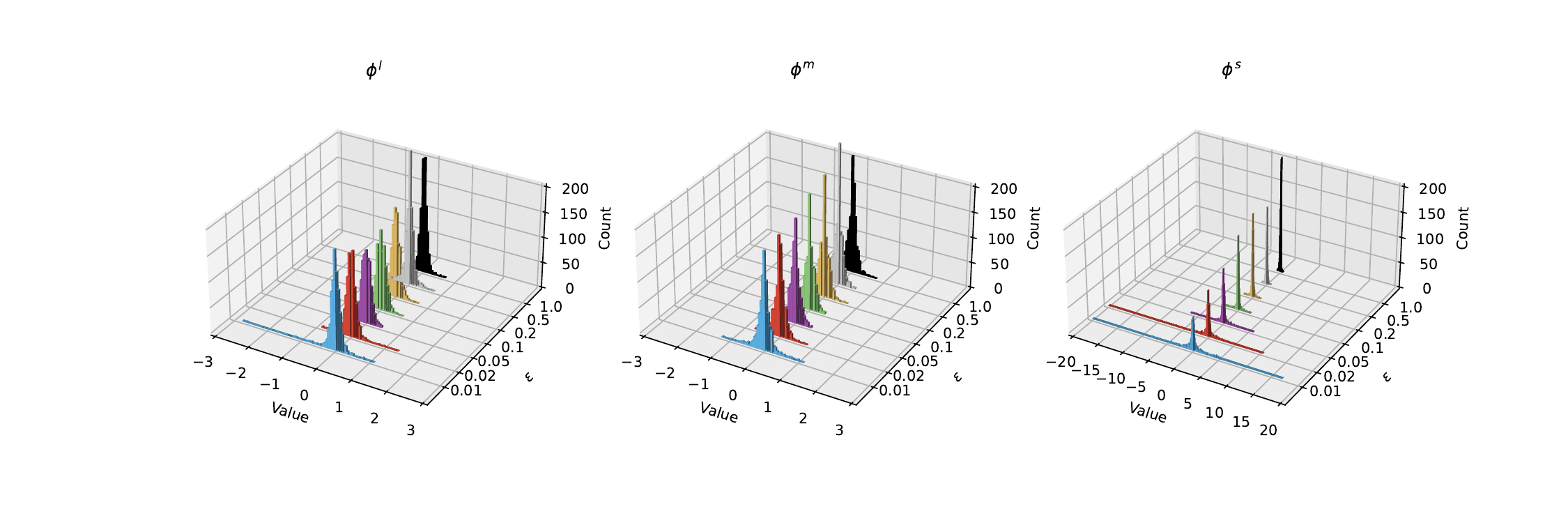}
    \caption{Histogram of model parameters for trained networks on a line segment with $c(x) = 1 + x^2$, $y = 0.95$, and varying $\varepsilon$. The parameter distributions for $\phi^l$ and $\phi^m$ remain relatively stable as $\varepsilon$ changes, whereas the parameter distribution for the small-scale network $\phi^s$ varies significantly. For small $\varepsilon$, $\phi^s$ requires larger parameters, making it more difficult to train.}
    \label{fig:seg_fixy_vlap_veps}
\end{figure}

As shown in the figure, the multiscale network $\phi^m$ exhibits consistent parameter distribution across all the $\varepsilon$ values, with minimal variation in the extreme values of the model parameters. The large-scale network alone demonstrates similar behavior. For large $\varepsilon$, the parameter distributions of the large-scale model $\phi^l$ are comparable to those of the multiscale network. However, when $\varepsilon$ becomes small (e.g., $\varepsilon=0.01$), the model parameters contain a few larger values compared to the multiscale network.
In contrast, a well-trained small-scale network $\phi^s$ requires many large parameter values to represent a good approximation when $\varepsilon$ is small, posing significant challenges for training. In fact, over $10^6$ iterations of Adam were required for $\phi^s$ to converge with $\varepsilon = 0.01$, while $20,000$ iterations is more than enough for $\phi^l$ and $\phi^m$. Finally, we note in passing that $\phi^l$ and $\phi^s$ are equivalent when $\varepsilon = 1.0$. 

\subsection{Test with Fixed $\mathbf{y}$}

In the previous sections, we justified the selection of $\varepsilon$ values in our experiments and studied the distribution of model parameters after the model converges.
In the following experiments, we focus on demonstrating that our multiscale network is easier to train compared to single scale models.

Our next group of experiments are still designed to test convergence with a fixed $\vy$. 
Specifically, we train the network to approximate $G(\vx,\vy)$ for $\vx\in\Omega$ and a single predetermined $\vy$. 
We start with the 1D problem \cref{eq:model problem} with $\Omega=[0,1]$ and $c=0$, $1$, and $1+x^2$, selecting $y$ values at 0.5, 0.7, and 0.95. 
These points were chosen to represent the center, typically internal points, and near-boundary points of the domain, respectively.
For the simple case where $c=0$, both the large-scale network $\phi^m_1$ and the small-scale network $\phi^m_2$ in the multiscale network $\phi^m$ utilize a hidden layer of $20$ neurons. 
The single scale networks $\phi^l$ and $\phi^s$ are equipped with a hidden layer comprising $40$ neurons.
For the cases where $c=1$ and $c=1+x^2$, we increase the number of neurons on the hidden layer to $100$ for $\phi^l$ and $\phi^s$, and set the number of neurons on the hidden layer to $50$ for $\phi^m_1$ and $\phi^m_2$.
During training, only one value of $y$ is sampled per test, along with two boundary points at $0$ and $1$.
Additionally, we uniformly sample $500$ points across $[0,1]$ and sample $500$ points near the selected $\vy$, as previously discussed.
The training regimen for $\phi^m$ involves $1,000$ steps for the large-scale network $\phi^m_1$, followed by $4,000$ steps for the entire network. 
In contrast, the baseline $\phi^l$ and $\phi^s$ undergoes a longer training of $5,000$ steps, resulting in a higher total computational cost, as training $\phi^m_1$ is much cheaper per step.
In this test, we set $\varepsilon=10^{-2}$.
For $c=1$ and $c=1+x^2$, we use a fine mesh to compute a high-accuracy FEM solution as the true solution.
In order to make a fair comparison, we apply the grid search on the initial learning rate and the decay rate for each model and report the result with the minimum error in the tests.
We search for the initial learning rate from $10^{-4}$ to $10^{-1}$, and search for the decay rate from $0.9$ to $1.0$ where the decay is applied every $500$ iterations.
These comparisons are depicted in \cref{fig:seg_fixy_lap slap and vlap}.

\begin{figure}[htbp]
    \centering
    \includegraphics[width=0.95\linewidth]{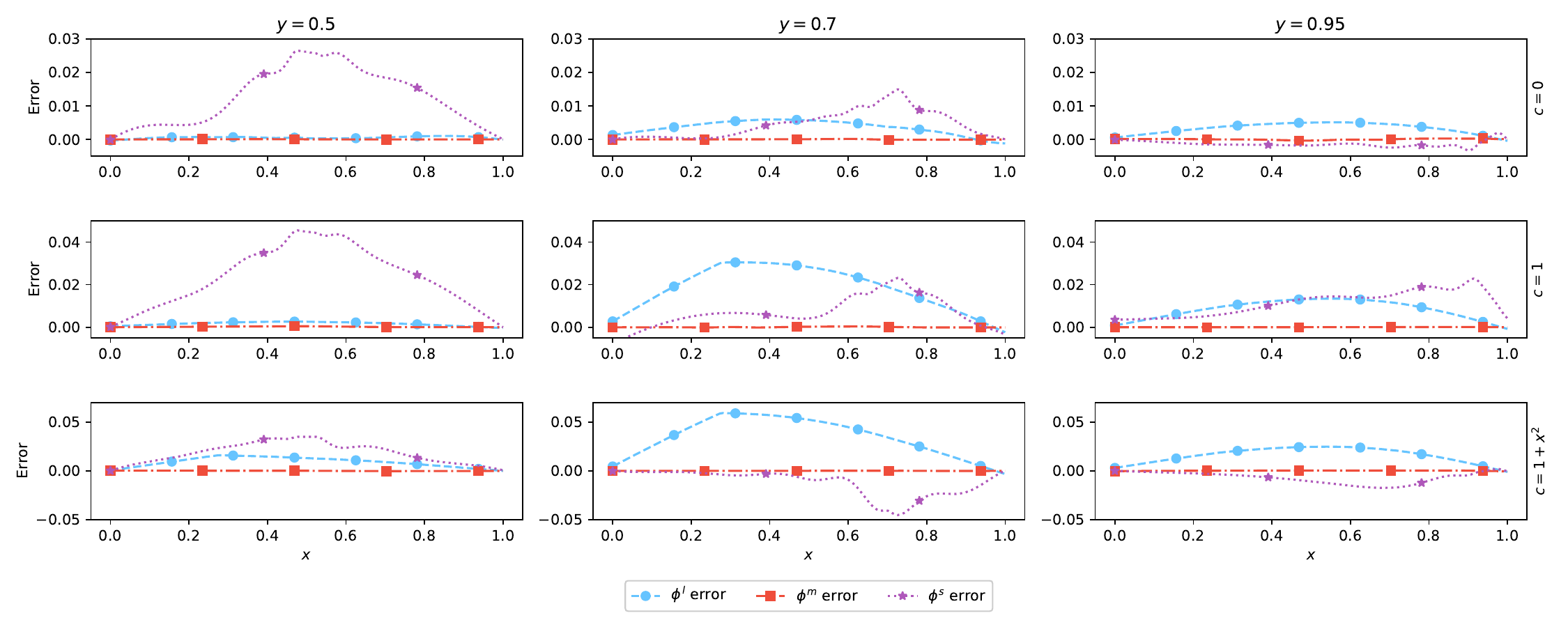}
    \caption{Approximation error comparison of three trained networks $\phi^l$, $\phi^m$ and $\phi^s$ on $\Omega=[0,1]$ at fixed $y$, with three cases of $c(x)=0$, $1$, and $1+x^2$. The reference solution is taken as a high-accuracy FEM solution computed on a fine mesh for $c(x)=1.0$ and $c(x)=1+x^2$. $\phi^l$, $\phi^m$ and $\phi^s$ have roughly the same number of parameters and are trained with the same number of epochs.}
    \label{fig:seg_fixy_lap slap and vlap}
\end{figure}

The figures show that the single-scale networks $\phi^l$ and $\phi^s$ not only incur higher training costs but also exhibit slower convergence compared to the multiscale network $\phi^m$. Notably, even after 5,000 iterations, the predictions of $\phi^s$ remain far from the exact Green's function. The standalone large-scale network $\phi^l$ performs well when $c = 0$ and $y = 0.5$, but its predictions worsen in other cases. In contrast, the multiscale network consistently delivers better predictions across all experiments at a lower cost. This evidence highlights the superior convergence capabilities of multiscale networks.

\begin{figure}[htbp]
    \centering
    \includegraphics[width=1.0\linewidth]{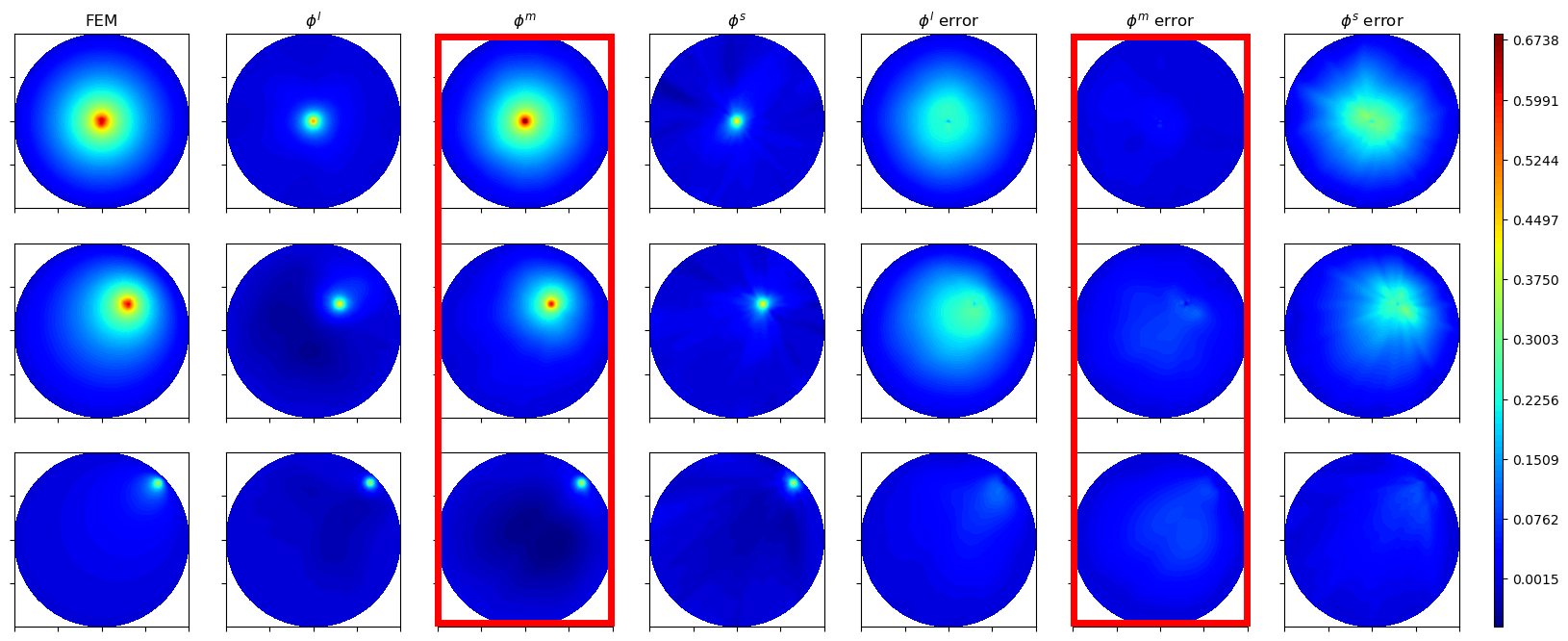}
    \caption{Approximation accuracy comparison for three trained networks $\phi^l$, $\phi^m$ and $\phi^s$ on a unit circle at three fixed $\vy=(0,0)$, $(0.3,0.3)$, and $(0.65,0.65)$ with $c(\vx)=1+x_1^2+x_2^2$. The reference solution is taken as a FEM solution computed on a fine mesh.     The results of $\phi^m$ are highlighted using red boxes. $\phi^l$, $\phi^m$ and $\phi^s$ have roughly the same number of parameters and are trained with the same number of epochs. }
    \label{fig:circ_fixy_vlap}
\end{figure}

Next, we conduct our tests in a 2-D unit circle domain with fixed points $\vy = (0, 0)$, $(0.3, 0.3)$, and $(0.65, 0.65)$. For this set of tests, we use $c(\vx) = 1 + x_1^2 + x_2^2$. Since this problem is more challenging, we employ networks with two hidden layers. The number of neurons in each hidden layer is set to 50 for $\phi^l$ and $\phi^s$, and 30 for $\phi^m_1$ and 40 for $\phi^m_2$. We sample 500 points on the boundary, 5,000 points near $\vy$, and 5,000 points uniformly across the entire domain.
For the MSNN, we first train $\phi^m_1$ for 2,000 steps, followed by training the full $\phi^m$ for 8,000 steps. For the single-scale networks, $\phi^l$ and $\phi^s$ are trained for 10,000 steps. The learning rate is set to $10^{-3}$ in this test. As shown in \cref{fig:circ_fixy_vlap}, our multiscale network consistently provides better predictions.

\subsection{Test on Solution of PDEs}

Finally, we evaluate our complete MSNN framework by solving PDEs. We revisit the 1D model problem on the interval $[0,1]$ with $c = 1 + x^2$. For this test, we partition the interval into 32 subdomains and set $\varepsilon = 0.01$. Each subdomain samples 30 different $y$ values; for each $y$, we associate 2 boundary points, 500 uniformly distributed points within the domain, and 500 points close to $y$.
We use networks with two hidden layers, with 50 neurons per layer for $\phi^l$ and $\phi^s$, 30 neurons for $\phi^m_1$, and 40 for $\phi^m_2$. The large-scale model $\phi^m_1$ is first trained for 4,000 iterations, followed by an additional 16,000 iterations to train the full multiscale network. In contrast, the baseline networks, $\phi^s$ and $\phi^l$, are trained for 20,000 iterations.
As shown in \cref{fig:test seg green and pde}, the results clearly demonstrate that the multiscale network achieves significantly better performance.

\begin{figure}[htb]
    \centering
    \includegraphics[width=0.95\textwidth]{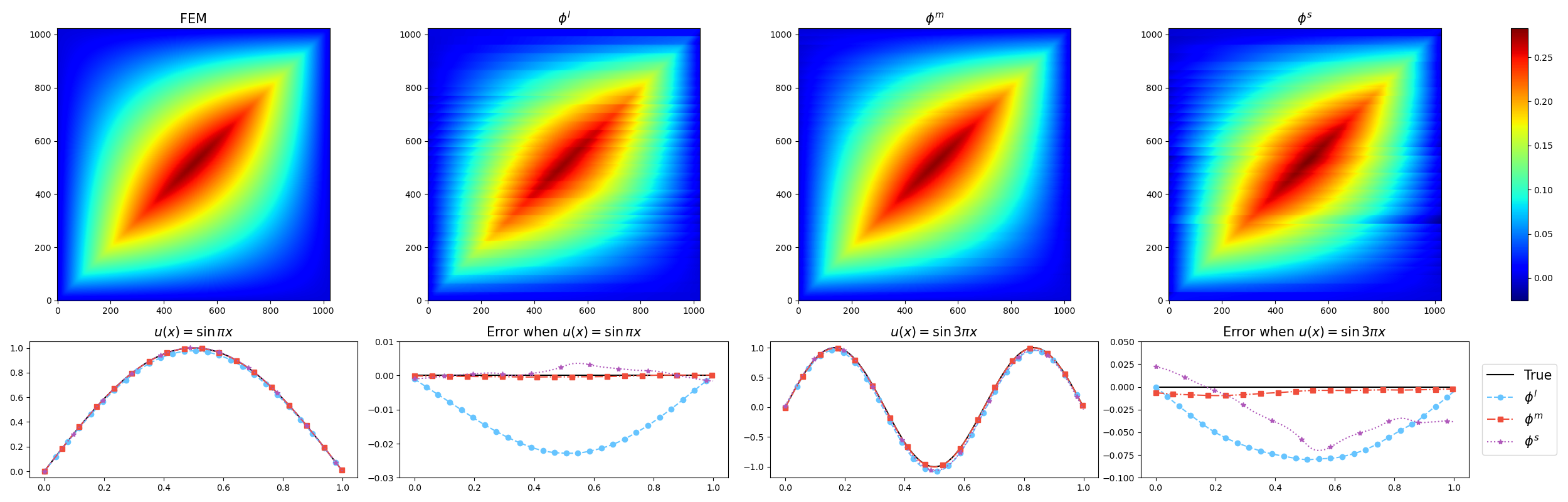}
    \caption{Predicted Green's function and numerical solutions of PDEs. Upper: comparison between predicted Green's function and a high accuracy FEM solution. Lower: comparison between solutions and error when solving PDEs with solutions $u(x) = \sin{(\pi x)}$ and $u(x) = \sin{(3\pi x)}$. Here we choose $c(x)=1+x^2$.}
    \label{fig:test seg green and pde}
\end{figure}

We further validate the accuracy of our approximation by solving PDEs with specified right-hand sides. Specifically, we test our approach using PDEs with solutions $u(x) = \sin{(\pi x)}$ and $u(x) = \sin{(3\pi x)}$. We then compare the exact solutions with the numerical solutions derived from the approximated Green’s function. As shown in \cref{fig:test seg green and pde}, our network's output provides significantly better approximations to the true solution in both cases.

Next, we extend our investigations to high-dimensional problems. In the upcoming experiment, we address 2D model problems on the domains $[-1, 1]^2$ and the unit circle, with $c(\vx) = 1 + x_1^2 + x_2^2$. For this test, the domain is divided into 32 subdomains, and $\varepsilon$ is set to 0.02.
We use networks with three hidden layers, with 50 neurons per layer for $\phi^l$ and $\phi^s$, 30 neurons for $\phi^m_1$, and 40 for $\phi^m_2$. When generating the datasets, 50 values of $\vy$ are selected, each associated with 20 boundary points, 200 uniformly distributed points within the domain, and 200 additional points near $\vy$.

Each large-scale model within the multiscale network $\phi^m$ is trained for 10,000 iterations, followed by 40,000 iterations for the full multiscale network. In comparison, the baseline single-scale networks $\phi^s$ and $\phi^l$ are trained for 50,000 iterations.
We assess the effectiveness of our approximation using PDEs with solutions $u(\vx) = \sin{(\pi x_1)}\sin{(\pi x_2)}$ and $u(\vx) = \sin{(3\pi x_1)}\sin{(3\pi x_2)}$ for the rectangular domain, and $u(\vx) = 1 - x_1^2 - x_2^2$ and $u(\vx) = (1 - x_1^2 - x_2^2) \sin{(3\pi x_1)}\sin{(3\pi x_2)}$ for the unit circle domain.
The results, shown in \cref{fig:test4 sol}, demonstrate that our network achieves a much closer approximation to the true solution with significantly lower training costs, highlighting the efficiency of our framework.

\begin{figure}[h!]
    \centering
    \includegraphics[width=0.96\linewidth]{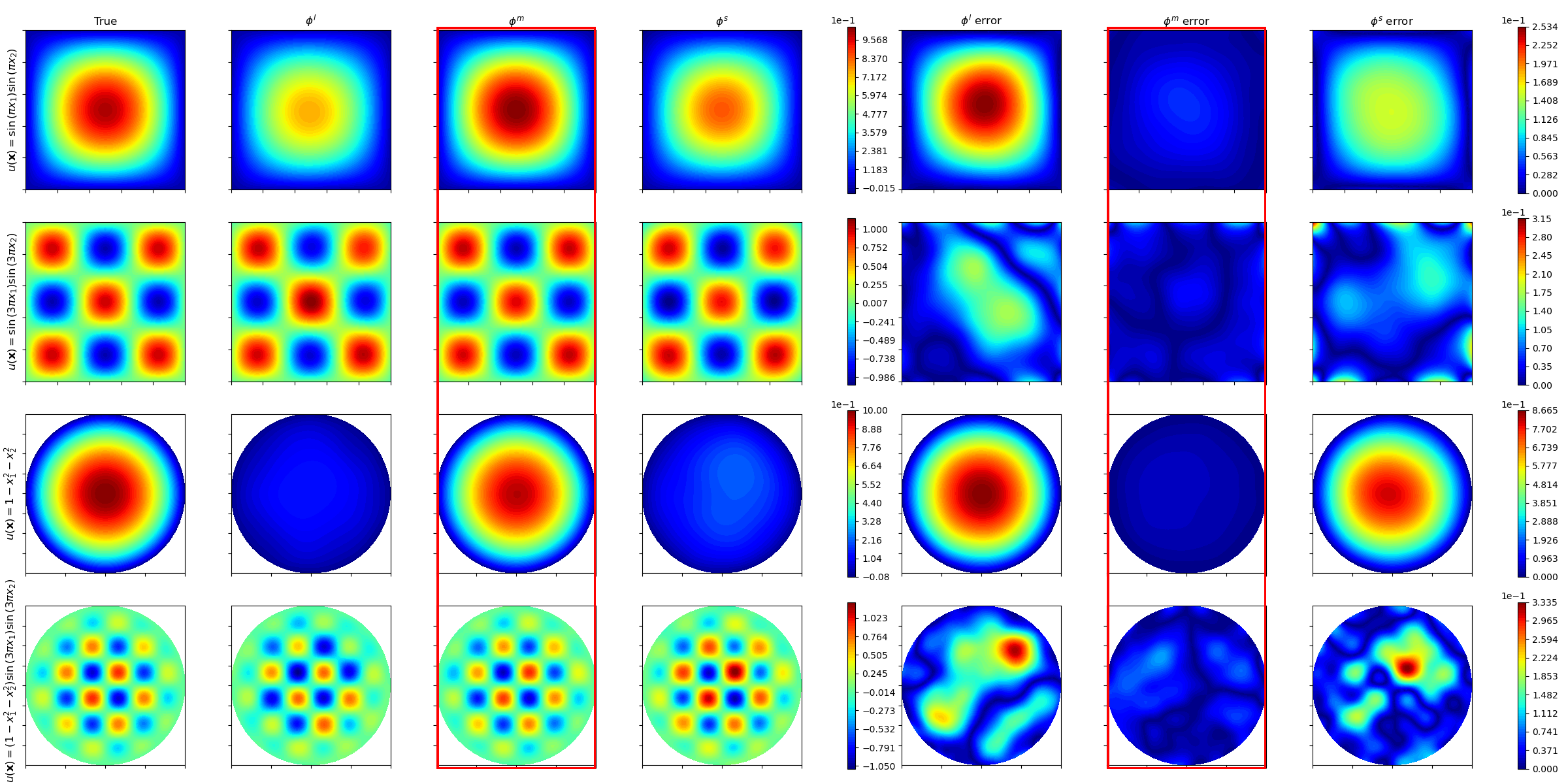}
    \caption{2D model problem with $c(\vx) = 1 + x_1^2 + x_2^2$. The exact solutions in the first column are $u(\vx) = \sin{(\pi x_1)}\sin{(\pi x_2)}$ and $u(\vx) = \sin{(3\pi x_1)}\sin{(3\pi x_2)}$ for the rectangular domain, and $u(\vx) = 1 - x_1^2 - x_2^2$ and $u(\vx) = (1 - x_1^2 - x_2^2)\sin{(3\pi x_1)}\sin{(3\pi x_2)}$ for the unit circle domain. Solutions obtained from different neural network frameworks are presented in the subsequent columns. The solutions from the multiscale network $\phi^m$ are highlighted with red boxes.}
    \label{fig:test4 sol}
\end{figure}

\section{Conclusion}
\label{conclusion}
In this paper, we propose a multiscale NN approach to learn Green's functions to solve PDE problems, consisting of both low-regularity and high-regularity components. We found that, compared to regular NNs, our approach requires fewer parameters and maintains the boundedness of the model parameters. Therefore, it can reduce training complexity and significantly speed up the training process, as demonstrated by both theoretical analysis and experimental performance. With the proposed approach, Green's functions 
can  be approximated
more effectively than a single-scale NN
because the large-scale NN can capture the large-scale (low-regularity) part and the small-scale neural networks is to approximate the residual high-regularity components.

While the theoretical results provide an existence proof for multiscale neural networks, a constructive proof would be valuable to verify the underlying assumptions and assist in designing neural networks for experimental purposes. We view this as a promising direction for future research. Furthermore, we intend to utilize the neural network approximation of the Green's function to develop more parallel and robust preconditioners than those based on sequential ILU factorizations \cite{tianshi2} for solving highly indefinite problems, such as high-frequency wave equations in our future work. {Finally, while this work focuses on a two‐scale architecture, the framework naturally extends to additional scales in settings where the Green’s function exhibits more complex multiscale structure.}
\section*{Acknowledgments}
The authors would like to thank Bei Hu and Qingguo Hong for many helpful discussions. 
The authors are also grateful to the two anonymous referees for their careful reading 
and valuable comments, which have helped improve the presentation of this paper.

\bibliographystyle{siamplain}
\bibliography{main}

@article{GREENGARD1987325,
title = {A fast algorithm for particle simulations},
journal = {Journal of Computational Physics},
volume = {73},
number = {2},
pages = {325-348},
year = {1987},
issn = {0021-9991},
doi = {https://doi.org/10.1016/0021-9991(87)90140-9},
url = {https://www.sciencedirect.com/science/article/pii/0021999187901409},
author = {L Greengard and V Rokhlin},
}

@article{gruter1982green,
  title={The Green function for uniformly elliptic equations},
  author={Gr{\"u}ter, Michael and Widman, Kjell-Ove},
  journal={Manuscripta mathematica},
  volume={37},
  number={3},
  pages={303--342},
  year={1982},
  publisher={Springer}
}

@article{li2024two,
  title={Two-layer networks with the {$\text{ReLU}^k$}  activation function: Barron spaces and derivative approximation},
  author={Li, Yuanyuan and Lu, Shuai and Math{\'e}, Peter and Pereverzev, Sergei V},
  journal={Numerische Mathematik},
  volume={156},
  number={1},
  pages={319--344},
  year={2024},
  publisher={Springer}
}

@article{gu2023stationary,
  title={Stationary density estimation of it{\^o} diffusions using deep learning},
  author={Gu, Yiqi and Harlim, John and Liang, Senwei and Yang, Haizhao},
  journal={SIAM Journal on Numerical Analysis},
  volume={61},
  number={1},
  pages={45--82},
  year={2023},
  publisher={SIAM}
}

@article{chen2022randomized,
  title={Randomized Newton’s method for solving differential equations based on the neural network discretization},
  author={Chen, Qipin and Hao, Wenrui},
  journal={Journal of Scientific Computing},
  volume={92},
  number={2},
  pages={49},
  year={2022},
  publisher={Springer}
}

@article{hao2024gauss,
  title={Gauss Newton method for solving variational problems of PDEs with neural network discretizaitons},
  author={Hao, Wenrui and Hong, Qingguo and Jin, Xianlin},
  journal={Journal of Scientific Computing},
  volume={100},
  number={1},
  pages={17},
  year={2024},
  publisher={Springer}
}

@article{chen2022bridging,
  title={Bridging traditional and machine learning-based algorithms for solving PDEs: the random feature method},
  author={Chen, J. and Chi, X. and Yang, Z. and others},
  journal={J Mach Learn},
  volume={1},
  pages={268--98},
  year={2022}
}

@article{dong2023method,
  title={A method for computing inverse parametric PDE problems with random-weight neural networks},
  author={Dong, S. and Wang, Y.},
  journal={Journal of Computational Physics},
  volume={489},
  pages={112263},
  year={2023},
  publisher={Elsevier}
}

@article{sun2024local,
  title={Local randomized neural networks with discontinuous Galerkin methods for partial differential equations},
  author={Sun, J. and Dong, S. and Wang, F.},
  journal={Journal of Computational and Applied Mathematics},
  volume={445},
  pages={115830},
  year={2024},
  publisher={Elsevier}
}

@article{yarotsky2017error,
  title={Error bounds for approximations with deep {ReLU} networks},
  author={Yarotsky, D.},
  journal={Neural Networks},
  volume={94},
  pages={103--114},
  year={2017},
  publisher={Elsevier}
}

@article{siegel2022optimal,
  title={Optimal Approximation Rates for Deep {ReLU} Neural Networks on {Sobolev} Spaces},
  author={Siegel, J.},
  journal={arXiv preprint arXiv:2211.14400},
  year={2022}
}

@article{weinan2018deep,
  title={{The Deep Ritz Method: A} Deep Learning-Based Numerical Algorithm for Solving Variational Problems},
  author={E, W. and Yu, B.},
  journal={Communications in Mathematics and Statistics},
  volume={6},
  number={1},
  year={2018},
  publisher={Springer Verlag}
}

@article{sirignano2018dgm,
  title={DGM: A deep learning algorithm for solving partial differential equations},
  author={Sirignano, J. and Spiliopoulos, K.},
  journal={Journal of computational physics},
  volume={375},
  pages={1339--1364},
  year={2018},
  publisher={Elsevier}
}

@article{li2020fourier,
  title={Fourier neural operator for parametric partial differential equations},
  author={Li, Z. and Kovachki, N. and Azizzadenesheli, K. and Liu, B. and Bhattacharya, K. and Stuart, A. and Anandkumar, A.},
  journal={arXiv preprint arXiv:2010.08895},
  year={2020}
}

@article{raissi2019physics,
  title={Physics-informed neural networks: A deep learning framework for solving forward and inverse problems involving nonlinear partial differential equations},
  author={Raissi, M. and Perdikaris, P. and Karniadakis, G.},
  journal={Journal of Computational Physics},
  volume={378},
  pages={686--707},
  year={2019},
  publisher={Elsevier}
}

@inproceedings{lu2021priori,
  title={A priori generalization analysis of the {Deep Ritz} method for solving high dimensional elliptic partial differential equations},
  author={Lu, J. and Lu, Y. and Wang, M.},
  booktitle={Conference on Learning Theory},
  pages={3196--3241},
  year={2021},
  organization={PMLR}
}

@article{luo2020two,
  title={Two-layer neural networks for partial differential equations: Optimization and generalization theory},
  author={Luo, T. and Yang, H.},
  journal={arXiv preprint arXiv:2006.15733},
  year={2020}
}

@book{anthony1999neural,
  title={Neural network learning: Theoretical foundations},
  author={Anthony, M. and Bartlett, P. and others},
  volume={9},
  year={1999},
  publisher={cambridge university press Cambridge}
}

@article{guhring2021approximation,
  title={Approximation rates for neural networks with encodable weights in smoothness spaces},
  author={G{\"u}hring, I. and Raslan, M.},
  journal={Neural Networks},
  volume={134},
  pages={107--130},
  year={2021},
  publisher={Elsevier}
}

@misc{
lin2021neural,
title={A neural network framework for learning Green's function},
author={Guochang Lin and Fukai Chen and Pipi Hu and Xiang Chen and Junqing Chen and Jun Wang and Zuoqiang Shi},
year={2022}
}

@article{zhang2022mod,
  title={MOD-Net: A Machine Learning Approach via Model-Operator-Data Network for Solving PDEs},
  author={Zhang, L. and Luo, T. and Zhang, Y. and John Xu, Z. and Ma, Z. and others},
  journal={Communications in Computational Physics},
  volume={32},
  number={2},
  pages={299--335},
  year={2022}
}

@article{fan2019multiscale,
  title={A multiscale neural network based on hierarchical matrices},
  author={Fan, Y. and Lin, L. and Ying, L. and Zepeda-N{\'u}nez, L.},
  journal={Multiscale Modeling \& Simulation},
  volume={17},
  number={4},
  pages={1189--1213},
  year={2019},
  publisher={SIAM}
}

@article{liu2020multi,
  title={Multi-scale deep neural network (MscaleDNN) for solving Poisson-Boltzmann equation in complex domains},
  author={Liu, Z. and Cai, W. and Xu, Z.},
  journal={arXiv preprint arXiv:2007.11207},
  year={2020}
}

@inproceedings{he2015delving,
  title={Delving deep into rectifiers: Surpassing human-level performance on imagenet classification},
  author={He, K. and Zhang, X. and Ren, S. and Sun, J.},
  booktitle={Proceedings of the IEEE international conference on computer vision},
  pages={1026--1034},
  year={2015}
}

@article{schmidt2020nonparametric,
author = {Johannes Schmidt-Hieber},
title = {{Nonparametric regression using deep neural networks with {ReLU} activation function}},
volume = {48},
journal = {The Annals of Statistics},
number = {4},
publisher = {Institute of Mathematical Statistics},
pages = {1875 -- 1897},
year = {2020},
}

@inproceedings{
yang2024deeper,
title={Deeper or Wider: A Perspective from Optimal Generalization Error with Sobolev Loss},
author={Yahong Yang and Juncai He},
booktitle={Forty-first International Conference on Machine Learning},
year={2024}
}

@article{suzuki2018adaptivity,
  title={Adaptivity of deep {ReLU} network for learning in Besov and mixed smooth Besov spaces: optimal rate and curse of dimensionality},
  author={Suzuki, T.},
  journal={arXiv preprint arXiv:1810.08033},
  year={2018}
}

@article{kingma2014adam,
  title={Adam: A method for stochastic optimization},
  author={Kingma, D. and Ba, J.},
  journal={arXiv preprint arXiv:1412.6980},
  year={2014}
}

@inproceedings{nah2017deep,
  title={Deep multi-scale convolutional neural network for dynamic scene deblurring},
  author={Nah, S. and Hyun Kim, T. and Mu Lee, K.},
  booktitle={Proceedings of the IEEE conference on computer vision and pattern recognition},
  pages={3883--3891},
  year={2017}
}

@article{wang2020multi,
  title={Multi-scale deep neural network (mscalednn) methods for oscillatory stokes flows in complex domains},
  author={Wang, B. and Zhang, W. and Cai, W.},
  journal={arXiv preprint arXiv:2009.12729},
  year={2020}
}

@article{siegel2023characterization,
  title={Characterization of the variation spaces corresponding to shallow neural networks},
  author={Siegel, J. and Xu, J.},
  journal={Constructive Approximation},
  volume={57},
  number={3},
  pages={1109--1132},
  year={2023},
  publisher={Springer}
}

@article{ma2020towards,
  title={Towards a mathematical understanding of neural network-based machine learning: what we know and what we don't},
  author={Ma, C. and Wojtowytsch, S. and Wu, L. and others},
  journal={arXiv preprint arXiv:2009.10713},
  year={2020}
}

@article{barron1993universal,
  title={Universal approximation bounds for superpositions of a sigmoidal function},
  author={Barron, A.},
  journal={IEEE Transactions on Information theory},
  volume={39},
  number={3},
  pages={930--945},
  year={1993},
  publisher={IEEE}
}

@article{siegel2022sharp,
  title={Sharp bounds on the approximation rates, metric entropy, and n-widths of shallow neural networks},
  author={Siegel, J. and Xu, J.},
  journal={Foundations of Computational Mathematics},
  pages={1--57},
  year={2022},
  publisher={Springer}
}

@article{bebendorf2003existence,
  title={Existence of $\mathcal{H}$-matrix approximants to the inverse FE-matrix of elliptic operators with ${L}^\infty$-coefficients},
  author={Bebendorf, M. and Hackbusch, W.},
  journal={Numerische Mathematik},
  volume={95},
  pages={1--28},
  year={2003},
  publisher={Springer}
}

@article{weinan2022representation,
  title={Representation formulas and pointwise properties for Barron functions},
  author={E, W. and Wojtowytsch, S.},
  journal={Calculus of Variations and Partial Differential Equations},
  volume={61},
  number={2},
  pages={46},
  year={2022},
  publisher={Springer New York}
}

@article{yang2023nearly,
  title={Nearly optimal approximation rates for deep super {RelU} networks on sobolev spaces},
  author={Yang, Y. and Wu, Y. and Yang, H. and Xiang, Y.},
  journal={arXiv preprint arXiv:2310.10766},
  year={2023}
}

@article{ji2023deep,
  title={Deep surrogate model for learning Green's function associated with linear reaction-diffusion operator},
  author={Ji, J. and Ju, L. and Zhang, X.},
  journal={arXiv preprint arXiv:2310.03642},
  year={2023}
}

@article{boulle2023learning,
  title={Learning elliptic partial differential equations with randomized linear algebra},
  author={Boull{\'e}, N. and Townsend, A.},
  journal={Foundations of Computational Mathematics},
  volume={23},
  number={2},
  pages={709--739},
  year={2023},
  publisher={Springer}
}

@inproceedings{teng2022learning,
  title={Learning green’s functions of linear reaction-diffusion equations with application to fast numerical solver},
  author={Teng, Y. and Zhang, X. and Wang, Z. and Ju, L.},
  booktitle={Mathematical and Scientific Machine Learning},
  pages={1--16},
  year={2022},
  organization={PMLR}
}

@article{zhou2013causal,
  title={Causal and structural connectivity of pulse-coupled nonlinear networks},
  author={Zhou, D. and Xiao, Y. and Zhang, Y. and Xu, Z. and Cai, D.},
  journal={Physical review letters},
  volume={111},
  number={5},
  pages={054102},
  year={2013},
  publisher={APS}
}

@book{adams2003sobolev,
  title={Sobolev spaces},
  author={Adams, R. and Fournier, J.},
  year={2003},
  publisher={Elsevier}
}

@article{yarotsky2020phase,
  title={The phase diagram of approximation rates for deep neural networks},
  author={Yarotsky, D. and Zhevnerchuk, A.},
  journal={Advances in neural information processing systems},
  volume={33},
  pages={13005--13015},
  year={2020}
}

@article{mhaskar1996neural,
  title={Neural networks for optimal approximation of smooth and analytic functions},
  author={Mhaskar, H.},
  journal={Neural computation},
  volume={8},
  number={1},
  pages={164--177},
  year={1996},
  publisher={MIT Press}
}

@article{majda2002vorticity,
  title={Vorticity and incompressible flow. Cambridge texts in applied mathematics},
  author={Majda, Andrew J and Bertozzi, Andrea L and Ogawa, A},
  journal={Appl. Mech. Rev.},
  volume={55},
  number={4},
  pages={B77--B78},
  year={2002}
}

@article{taylor2013green,
  title={The Green function for elliptic systems in two dimensions},
  author={Taylor, Justin L and Kim, Seick and Brown, Russell Murray},
  journal={Communications in Partial Differential Equations},
  volume={38},
  number={9},
  pages={1574--1600},
  year={2013},
  publisher={Taylor \& Francis}
}

@article{hofmann2007green,
  title={The Green function estimates for strongly elliptic systems of second order},
  author={Hofmann, Steve and Kim, Seick},
  journal={manuscripta mathematica},
  volume={124},
  number={2},
  pages={139--172},
  year={2007},
  publisher={Springer}
}

@article{kim2019green,
  title={Green’s function for second order elliptic equations with singular lower order coefficients},
  author={Kim, S. and Sakellaris, G.},
  journal={Communications in Partial Differential Equations},
  volume={44},
  number={3},
  pages={228--270},
  year={2019},
  publisher={Taylor \& Francis}
}

@article{siegel2023greedy,
  title={Greedy training algorithms for neural networks and applications to PDEs},
  author={Siegel, J. and Hong, Q. and Jin, X. and Hao, W. and Xu, J.},
  journal={Journal of Computational Physics},
  volume={484},
  pages={112084},
  year={2023},
  publisher={Elsevier}
}

@article{ying2012pedestrian,
  title={A pedestrian introduction to fast multipole methods},
  author={Ying, Lexing},
  journal={Science China Mathematics},
  volume={55},
  pages={1043--1051},
  year={2012},
  publisher={Springer}
}

@article{beatson1997short,
  title={A short course on fast multipole methods},
  author={Beatson, Rick and Greengard, Leslie and others},
  journal={Wavelets, multilevel methods and elliptic PDEs},
  volume={1},
  number={1},
  year={1997},
  publisher={Leicester}
}

@article{yang2023homotopy,
  title={Homotopy relaxation training algorithms for infinite-width two-layer {ReLU} neural networks},
  author={Yang, Yahong and Chen, Qipin and Hao, Wenrui},
  journal={Journal of scientific computing},
  volume={102},
  number={2},
  pages={40},
  year={2025},
  publisher={Springer}
}

@article{ainsworth2022active,
  title={Active Neuron Least Squares: A training method for multivariate rectified neural networks},
  author={Ainsworth, M. and Shin, Y.},
  journal={SIAM Journal on Scientific Computing},
  volume={44},
  number={4},
  pages={A2253--A2275},
  year={2022},
  publisher={SIAM}
}

@article{howard2023multifidelity,
  title={Multifidelity deep operator networks for data-driven and physics-informed problems},
  author={Howard, A. and Perego, M. and Karniadakis, G. and Stinis, P.},
  journal={Journal of Computational Physics},
  volume={493},
  pages={112462},
  year={2023},
  publisher={Elsevier}
}

@article{zheng2024hompinns,
  title={HomPINNs: Homotopy physics-informed neural networks for solving the inverse problems of nonlinear differential equations with multiple solutions},
  author={Zheng, H. and Huang, Y. and Huang, Z. and Hao, W. and Lin, G.},
  journal={Journal of Computational Physics},
  volume={500},
  pages={112751},
  year={2024},
  publisher={Elsevier}
}

@article{gin2021deepgreen,
  title={DeepGreen: deep learning of Green’s functions for nonlinear boundary value problems},
  author={Gin, C. and Shea, D. and Brunton, S. and Kutz, J N.},
  journal={Scientific reports},
  volume={11},
  number={1},
  pages={21614},
  year={2021},
  publisher={Nature Publishing Group UK London}
}

@article{wimalawarne2023learning,
  title={Learning {Green's} Function Efficiently Using Low-Rank Approximations},
  author={Wimalawarne, K. and Suzuki, T. and Langer, S.},
  journal={arXiv preprint arXiv:2308.00350},
  year={2023}
}

@inproceedings{weinan2022some,
  title={Some observations on high-dimensional partial differential equations with barron data},
  author={E, W. and Wojtowytsch, S.},
  booktitle={Mathematical and Scientific Machine Learning},
  pages={253--269},
  year={2022},
  organization={PMLR}
}

@book{shalev2014understanding,
  title={Understanding machine learning: From theory to algorithms},
  author={Shalev-Shwartz, S. and Ben-David, S.},
  year={2014},
  publisher={Cambridge university press}
}

@article{gu2021structure,
  title={Structure probing neural network deflation},
  author={Gu, Y. and Wang, C. and Yang, H.},
  journal={Journal of Computational Physics},
  volume={434},
  pages={110231},
  year={2021},
  publisher={Elsevier}
}

@misc{kloeckner2018meshpy,
  title={MeshPy: Simplicial Mesh Generation from Python},
  author={Kloeckner, A},
  year={2018}
}

@article{dunavant1985high,
  title={High degree efficient symmetrical Gaussian quadrature rules for the triangle},
  author={Dunavant, DA794241},
  journal={International journal for numerical methods in engineering},
  volume={21},
  number={6},
  pages={1129--1148},
  year={1985},
  publisher={Wiley Online Library}
}

@article{alnaes2015fenics,
  title={The FEniCS project version 1.5},
  author={Aln{\ae}s, Martin and Blechta, Jan and Hake, Johan and Johansson, August and Kehlet, Benjamin and Logg, Anders and Richardson, Chris and Ring, Johannes and Rognes, Marie E and Wells, Garth N},
  journal={Archive of numerical software},
  volume={3},
  number={100},
  year={2015}
}

@article{montanelli2019new,
  title={New error bounds for deep {ReLU} networks using sparse grids},
  author={Montanelli, H. and Du, Q.},
  journal={SIAM Journal on Mathematics of Data Science},
  volume={1},
  number={1},
  pages={78--92},
  year={2019},
  publisher={SIAM}
}

@book{evans2022partial,
  title={Partial differential equations},
  author={Evans, L.},
  volume={19},
  year={2022},
  publisher={American Mathematical Society}
}

@book{logg2012automated,
  title={Automated solution of differential equations by the finite element method: The FEniCS book},
  author={Logg, Anders and Mardal, Kent-Andre and Wells, Garth},
  volume={84},
  year={2012},
  publisher={Springer Science \& Business Media}
}

@article{chen2021representation,
  title={On the representation of solutions to elliptic pdes in barron spaces},
  author={Chen, Z. and Lu, J. and Lu, Y.},
  journal={Advances in neural information processing systems},
  volume={34},
  pages={6454--6465},
  year={2021}
}

@article{wojtowytsch2022representation,
  title={Representation formulas and pointwise properties for Barron functions},
  author={Wojtowytsch, S. and E, W.},
  journal={Calculus of Variations and Partial Differential Equations},
  volume={61},
  number={2},
  pages={1--37},
  year={2022},
  publisher={Springer}
}

@article{ma2022barron,
  title={The Barron space and the flow-induced function spaces for neural network models},
  author={Ma, C. and Wu, L. and others},
  journal={Constructive Approximation},
  volume={55},
  number={1},
  pages={369--406},
  year={2022},
  publisher={Springer}
}

@inproceedings{NEURIPS2019_bdbca288,
 author = {Paszke, Adam and Gross, Sam and Massa, Francisco and Lerer, Adam and Bradbury, James and Chanan, Gregory and Killeen, Trevor and Lin, Zeming and Gimelshein, Natalia and Antiga, Luca and Desmaison, Alban and Kopf, Andreas and Yang, Edward and DeVito, Zachary and Raison, Martin and Tejani, Alykhan and Chilamkurthy, Sasank and Steiner, Benoit and Fang, Lu and Bai, Junjie and Chintala, Soumith},
 booktitle = {Advances in Neural Information Processing Systems},
 editor = {H. Wallach and H. Larochelle and A. Beygelzimer and F. d\textquotesingle Alch\'{e}-Buc and E. Fox and R. Garnett},
 pages = {},
 publisher = {Curran Associates, Inc.},
 title = {PyTorch: An Imperative Style, High-Performance Deep Learning Library},
 volume = {32},
 year = {2019}
}

@Article{         harris2020array,
 title         = {Array programming with {NumPy}},
 author        = {Charles R. Harris and K. Jarrod Millman and St{\'{e}}fan J.
                 van der Walt and Ralf Gommers and Pauli Virtanen and David
                 Cournapeau and Eric Wieser and Julian Taylor and Sebastian
                 Berg and Nathaniel J. Smith and Robert Kern and Matti Picus
                 and Stephan Hoyer and Marten H. van Kerkwijk and Matthew
                 Brett and Allan Haldane and Jaime Fern{\'{a}}ndez del
                 R{\'{i}}o and Mark Wiebe and Pearu Peterson and Pierre
                 G{\'{e}}rard-Marchant and Kevin Sheppard and Tyler Reddy and
                 Warren Weckesser and Hameer Abbasi and Christoph Gohlke and
                 Travis E. Oliphant},
 year          = {2020},
 month         = sep,
 journal       = {Nature},
 volume        = {585},
 number        = {7825},
 pages         = {357--362},
 publisher     = {Springer Science and Business Media {LLC}},
}

@article{tianshi,
author = {Xu, Tianshi and Kalantzis, Vassilis and Li, Ruipeng and Xi, Yuanzhe and Dillon, Geoffrey and Saad, Yousef},
title = {         parGeMSLR: A parallel multilevel Schur complement low-rank preconditioning and solution package for general sparse matrices},
year = {2022},
issue_date = {Oct 2022},
publisher = {Elsevier Science Publishers B. V.},
address = {NLD},
volume = {113},
number = {C},
issn = {0167-8191},
journal = {Parallel Comput.},
}

@article{tianshi2,
  title={A two-level GPU-accelerated incomplete LU preconditioner for general sparse linear systems},
  author={Xu, Tianshi and Li, Ruipeng and Osei-Kuffuor, Daniel},
  journal={arXiv preprint arXiv:2303.08881},
  year={2023}
}

@article{Lu2019LearningNO,
  title={Learning nonlinear operators via DeepONet based on the universal approximation theorem of operators},
  author={Lu Lu and Pengzhan Jin and Guofei Pang and Zhongqiang Zhang and George Em Karniadakis},
  journal={Nature Machine Intelligence},
  year={2019},
  volume={3},
  pages={218 - 229},
}

@article{xu2024parallel,
  title={A Parallel Algorithm for Computing Partial Spectral Factorizations of Matrix Pencils via Chebyshev Approximation},
  author={Xu, Tianshi and Austin, Anthony and Kalantzis, Vasileios and Saad, Yousef},
  journal={SIAM Journal on Scientific Computing},
  volume={46},
  number={2},
  pages={S324--S351},
  year={2024},
  publisher={SIAM}
}

@article{nltgcr,
author = {He, Huan and Tang, Ziyuan and Zhao, Shifan and Saad, Yousef and Xi, Yuanzhe},
title = {nlTGCR: A Class of Nonlinear Acceleration Procedures Based on Conjugate Residuals},
journal = {SIAM Journal on Matrix Analysis and Applications},
volume = {45},
number = {1},
pages = {712-743},
year = {2024}
}

@InProceedings{KingBa15,
  author    = {Kingma, Diederik and Ba, Jimmy},
  booktitle = {International Conference on Learning Representations (ICLR)},
  title     = {Adam: A Method for Stochastic Optimization},
  year      = {2015},
  address   = {San Diega, CA, USA},
  optmonth  = {12},
}

@article{ddh2,
author = {Cai, Difeng and Huang, Hua and Chow, Edmond and Xi, Yuanzhe},
title = {Data-Driven Construction of Hierarchical Matrices With Nested Bases},
journal = {SIAM Journal on Scientific Computing},
volume = {46},
number = {2},
pages = {S24-S50},
year = {2024},
}

@article{smash,
author = {Cai, Difeng and Chow, Edmond and Erlandson, Lucas and Saad, Yousef and Xi, Yuanzhe},
title = {SMASH: Structured matrix approximation by separation and hierarchy},
journal = {Numerical Linear Algebra with Applications},
volume = {25},
number = {6},
pages = {e2204},
note = {e2204 nla.2204},
year = {2018}
}

@book{hbook,
author = {Hackbusch, Wolfgang},
title = {Hierarchical Matrices: Algorithms and Analysis},
year = {2015},
isbn = {3662473232},
publisher = {Springer Publishing Company, Incorporated},
edition = {1st},
}

\end{document}